\documentclass[11pt]{article}

\usepackage{amsfonts,amssymb,amsmath,amsthm}
\usepackage{enumerate}
\usepackage{url}
\usepackage{enumerate}
\usepackage{color}
\usepackage{algorithmic}
\usepackage{algorithm}
\usepackage[colorlinks]{hyperref}
\usepackage[square,authoryear]{natbib}

\theoremstyle{plain}
\newtheorem{theorem}{Theorem}[section]
\newtheorem{corollary}[theorem]{Corollary}
\newtheorem{proposition}[theorem]{Proposition}
\newtheorem{lemma}[theorem]{Lemma}

\theoremstyle{definition}
\newtheorem{definition}[theorem]{Definition}

\theoremstyle{remark}
\newtheorem{remark}[theorem]{Remark}

\newcommand{\Proj}{\textnormal{Proj}\,}

\newcommand{\hilb}{{\textnormal{Hilb}}}
\newcommand{\hilbp}{{\textnormal{Hilb}_{p(t)}^n}}
\newcommand{\Mf}{\mathrm{Mf}}
\newcommand{\id}{\mathfrak}
\newcommand{\cN}{\mathcal{N}}
\newcommand{\Ht}{\textnormal{Ht}}

\newcommand{\supp}{\mathrm{Supp}}

\newcommand{\PP}{\mathbb{P}}

\numberwithin{equation}{section}

\title 
{Smoothable Gorenstein points via marked schemes and double-generic initial ideals}

\author{
\textsc{Cristina Bertone}\footnote{https://sites.google.com/view/cristinabertone/home}\\ 
\small Dipartimento di Matematica dell'Universit\`a di Torino,  Torino, Italy \\ 
\small \href{mailto:cristina.bertone@unito.it}{cristina.bertone@unito.it}
\and
\textsc{Francesca Cioffi}\\
\small Dipartimento di Matematica e Applicazioni dell'Universit\`a di Napoli Federico II, Napoli, Italy\\
\small \href{mailto:cioffifr@unina.it}{cioffifr@unina.it}
\and
\textsc{Margherita Roggero}\\
\small Dipartimento di Matematica dell'Universit\`a di Torino, Torino, Italy\\
\small \href{mailto:margherita.roggero@unito.it}{margherita.roggero@unito.it}}
\date{}

\begin{document}

\maketitle

\begin{abstract}
Over an infinite field $K$ with $\mathrm{char}(K)\neq 2,3$, we investigate smoothable Gorenstein $K$-points in a punctual Hilbert scheme and obtain the following results: (i)~every $K$-point defined by local Gorenstein $K$-algebras with Hilbert function $(1,7,7,1)$ is smoothable (this is the only case non treated in the range considered by Iarrobino and Kanev in 1999); (ii)~the Hilbert scheme $\hilb_{16}^7$ has at least $5$ irreducible components. As a byproduct of our study about $\hilb_{16}^7$ we also find a new elementary component in $\hilb_{15}^7$.

We face the problem from a new point of view, that is based on properties of double-generic initial ideals and of marked schemes. 
 The properties of marked schemes give us a simple method to compute the Zariski tangent space to a Hilbert scheme at a given $K$-point, which is very useful in this context. We also test our tools to find the already known result that $K$-points defined by local Gorenstein $K$-algebras with Hilbert function $(1,5,5,1)$ are smoothable. The problem that we consider is strictly related to the study of the irreducibility of the Gorenstein locus in a Hilbert scheme and, more generally, of the irreducibility of a Hilbert scheme, which is a very open question.
\end{abstract}

{\let\thefootnote\relax\footnote{{{\it Keywords:} Gorenstein  algebra, Hilbert scheme, strongly stable ideal, marked basis; 
\noindent{\it MSC 2010:} 13H10,14C05, 13P10.}}}

\section*{Introduction}

This paper is devoted to investigating when Gorenstein points in a punctual Hilbert scheme are smoothable. In particular, we are interested in studying Gorenstein points defined by graded (Artin) $K$-algebras with Hilbert function $(1,7,7,1)$, that is the only case not treated in the range considered in \citep[Lemma 6.21]{IK99} for the detection of nonsmoothable points in a punctual Hilbert scheme in characteristic $0$. More precisely, over any infinite field $K$ with characteristic other than $2$ and $3$ we prove that:
\begin{itemize}
\item[(i)] Every local Gorenstein $K$-algebras with Hilbert function $(1,7,7,1)$ is smoothable (see Theorem \ref{Gor7}).
\item[(ii)] There are at least five irreducible components in $\hilb_{16}^7$  (see Section \ref{IrrComp7}).
\end{itemize}
We also show how our arguments apply to prove the now known result that local Gorenstein $K$-algebras with Hilbert function $(1,5,5,1)$ are smoothable (see Theorem \ref{Gor5}). 

We obtain our results facing the problem from a new point of view: we apply the notion of double-generic initial ideal, which has a crucial role in the detection of irreducible components and of generic initial ideals (see \citep{BCR-glin}), and constructive methods that are based on marked schemes (see \citep{BCR} and the references therein) and are useful to compute the Zariski tangent space to a Hilbert scheme at a given point (see Corollary \ref{cor:tangente} and Remark \ref{rem:tangente}). Step by step we alternate experimental results and theoretical properties of our tools, which have been applied in this context for the first time. Throughout the description of our experimental results we often refer to ancillary files that are available  as supplementary material of this paper, and at the webpage of the first author. 
The ancillary files are organized in folders, whose numbering is the same as the numbering of theorems, propositions and lemmas the reader finds in the paper. 

To address the problem, let $K$ be an infinite field of characteristic other than $2$ and $3$. For any positive integer $n$ and an admissible Hilbert polynomial $p(t)$, let $\hilbp$ denote the Hilbert scheme parameterizing the projective subschemes of $\mathbb P^n_K$ with Hilbert polynomial $p(t)$. We deal with punctual Hilbert schemes, hence with constant Hilbert polynomials, and when we take a point we mean a $K$-valued point ($K$-point, for short), i.e.~a closed point with residue field $K$.

Let $p(t)=d$ be the Hilbert polynomial of $d$ points. The {\em smoothable component $\mathcal R_d^n$ of $\hilbp$} is the closure of the open set of points corresponding to ideals of $d$ distinct points, i.e.~the rational component of $\hilbp$ containing the point corresponding to the lex-segment ideal.

A zero-dimensional subscheme $X$ is {\em smoothable} if it belongs to the smoothable component $\mathcal R_d^n$ or, equivalently, the $K$-algebra $A$ defining $X=\Proj (A)$ is isomorphic to the special fiber of a flat one-parameter family of $K$-algebras with smooth general point (e.g.~\citep[Lemma 4.1]{CEVV}, see also \citep[Definitions 5.16 and 6.20]{IK99}). 

As noted in \citep[Remark 1.6]{CEVV} for Hilbert schemes of points, every point in $\hilbp$ has an open neighborhood that can be studied by suitable \lq\lq affine\rq\rq\ techniques. In the same context, a similar approach is also used in \citep[Chapter 18]{MS}. So, up to a suitable change of coordinates, we can identify every point of a punctual Hilbert scheme $\hilb^n_{d}$ with an ideal in $R:=K[x_1,\dots,x_n]$, non-necessarily homogeneous, with affine Hilbert polynomial $p(t)=d$. 

A zero-dimensional subscheme $X$ is smoothable if and only if the same is true for all its irreducible components (e.g.~\citep[page 1245]{CN10}, \citep[Section 4]{CEVV}). This observation motivates interest for the so-called {\em elementary components} of a punctual Hilbert scheme, i.e.~components whose  points parameterize  zero-dimensional subschemes with support of cardinality one (see \citep{JJ-arxiv,Hu} for recent contributions in this context). Hence, the problem of detecting  smoothable points is connected to the study of ideals $I\subset R$ such that $R/I$ is a local $K$-algebra. 

A graded Artin $K$-algebra is necessarily local, in particular defines a scheme supported on a single point.  If the Hilbert function is of type $(1,n,n,1)$ then also the converse holds, that is every local Artin Gorenstein $K$-algebra with Hilbert function $(1,n,n,1)$ is graded (see \citep[Proposition 1.1]{JJ2017} and, on algebraically closed field of characteristic $0$, see \citep[Theorem 3.3]{ER2012}). Hence, we can state our results for local Artin Gorenstein $K$-algebras although we will prove them for graded Artin Gorenstein $K$-algebras. Moreover, the result that an Artin $K$-algebra $A$ is smoothable if and only if $A\otimes_K \bar K$ is smoothable, where $\bar K$ is the algebraic closure of $K$ \citep[Proposition 1.2]{JJB2017}, will allow us to also consider a field $K$ that is not algebraically closed.  

The study of Gorenstein smoothable points is strictly related to the study of the irreducibility of the Gorenstein locus in a Hilbert scheme. In this context,  it is well-known that a punctual Hilbert scheme $\hilb_d^n$ is irreducible if $n=2$ (see \citep{Fo}) and if $d\leq 7$ for $n\geq 3$ (see \citep{CEVV}). Moreover, the Gorenstein locus of $\hilb_d^n$ is irreducible if $d\leq 13$ for every $n$ (see \citep{CN9,CN10,CN14,CN15} and the references therein). Other relevant and also more general results about irreducibility in a Hilbert scheme are due to Ellingsrud and Iarrobino. 

An outline of some results of the present paper was described by ansatz in \citep{BCR-arxiv} as an application of the constructive methods about marked bases in an affine framework that were lately deeply studied and completely described in \citep{BCR}. Here, we give an extensive description of the outcome of our study. Different proofs of the case $(1,5,5,1)$ were presented in \citep{JJ}, contemporary to our first version given in \citep{BCR-arxiv}, and later in \citep{CN15} when $\mathrm{char}(K)\neq 2,3$.

The paper is organized in the following way. In Section \ref{sec:marked}, we describe the results about marked schemes that we need in our arguments. These results also give information on the computation of the Zariski tangent space to a Hilbert scheme via marked schemes (Corollary \ref{cor:tangente}). In Section \ref{sec:gore}, we focus on Gorenstein schemes and their relation with double-generic initial ideals when the Hilbert function is of type $(1,n,n,1)$ (Proposition \ref{prop:GLstable}).  In Sections \ref{GorSmooth} and \ref{GorSmooth2} we prove that local Artin Gorenstein $K$-algebras with Hilbert function either $(1,7,7,1)$ or $(1,5,5,1)$ are smoothable. In Section \ref{IrrComp7}, we prove that $\hilb_{16}^7$ has at least $5$ irreducible components in the following way: we explicitly describe two new components in $\hilb_{16}^7$, and then observe that they differ from the smoothable component $\mathcal R^7_{16}$ and from other two irreducible components $\mathcal{IE}_1$ and $\mathcal{IE}_2$ that we investigate starting from already known elementary components of $\hilb_8^4$ and $\hilb_8^7$ (see \citep{IE1978} and \citep{CEVV}). The anonymous referee highlighted us the existence of $\mathcal{IE}_1$ and $\mathcal{IE}_2$. Throughout this study we also prove the existence of an elementary component in $\hilb_{15}^7$ (see Corollary \ref{cor:elementary component}).

We distinguish the five components thanks to computations of the dimension of suitable tangent spaces to the Hilbert scheme. We also observe that all these five components contain the point corresponding to a certain strongly stable ideal that has a crucial role in our investigation.

\section{Background: marked schemes and Zariski tangent space}
\label{sec:marked}

In this paper, we work in an affine framework and apply the affine computational techniques developed in \citep{BCR}. Then, in this section, we set some notations and recall the main notions involved in these techniques. Moreover, we give some new insights for the computation of the Zariski tangent space to a Hilbert scheme at a given point. {\em Here, we can take any field $K$}.

We will consider the rings of polynomials $R=K[x_1,\dots,x_n]\subset S=R[x_0]$, with $x_0<x_1<\dots <x_n$. For a term $x^\alpha=x_0^{\alpha_0}x_1^{\alpha_1}\dots x_n^{\alpha_n}$ we set $\vert\alpha\vert:=\sum_i \alpha_i$, $\max(x^\alpha):=\max\{x_i\ \vert\ \alpha_i\neq 0\}$ and $\min(x^\alpha):=\min\{x_i \ \vert\ \alpha_i\neq 0\}$. For a non-null polynomial $f$ we denote by $\supp(f)$ its {\em support}, that is the set of terms that appear in $f$ with a non-null coefficient. 
If $f$ is a polynomial in $R$ then we denote by $f^h:=x_0^{deg(f)}f(\frac{x_1}{x_0},\dots,\frac{x_n}{x_0})$ its homogenization, and if $F$ is a polynomial in $S$ then we denote by $F^a:=F(1,x_1,\dots,x_n)$ its dehomogenization. 

Given a monomial ideal $\mathfrak j\subset R$ (resp.~$J\subset S$), we denote by $\mathcal N(\mathfrak j)$ (resp.~$\mathcal N(J)$)  the set of terms of $R$ outside $\mathfrak j$ (resp.~of $S$ outside $J$) and by $B_{\mathfrak j}$ (resp.~$B_J$) its minimal monomial basis. We refer to \citep{KR,KR2,Mora} for results concerning Gr\"obner bases and Hilbert functions.

For any set $\Gamma$ of polynomials, we will denote by $\Gamma_t$ the set of homogeneous polynomials of $\Gamma$ of degree $t$, by $\Gamma_{\leq t}$ the set of polynomials of $\Gamma$ of degree $\leq t$.  Furthermore, we denote by $\langle \Gamma\rangle$ the $A$-module generated by $\Gamma$. 

The results we are going to recall are based on the special features of \emph{quasi-stable}, \emph{stable} and \emph{strongly stable} ideals. We give the definition for ideals in $S$, but it obviously apply also in $R$.

\begin{definition}
Let $J$ be a monomial ideal in $S$.\\
The ideal $J$ is \emph{quasi stable} if for every $x^\alpha\in J$, for every $x_j>\min(x^\alpha)$, there is $s\geq 0$ such that $x_j^sx^\alpha/\min(x^\alpha)$ belongs to $J$.\\
The ideal $J$ is \emph{stable} if for every $x^\alpha\in J$, for every $x_j>\min(x^\alpha)$, $x_jx^\alpha/\min(x^\alpha)$ belongs to $J$.\\
The ideal $J$ is \emph{strongly stable} if for every $x^\alpha \in J$, for every $x_i$ dividing $x^\alpha$ and for every $x_j>x_i$, $x_jx^\alpha/x_i$ belongs to $J$.
\end{definition}

Obviously, strongly stable ideals are also stable, and stable ideals are also quasi-stable.

We are interested in monomial ideals having a (non-minimal) special set of monomial generators that is called \emph{Pommaret basis}.

\begin{definition}\citep{Seiler2009I}\label{def:Pommaret}
The \emph{Pommaret cone} of a term $x^\alpha$ is the set 
$$\mathcal C_{\mathcal P}(x^\alpha)=\lbrace x^\delta x^\alpha \vert \max(x^\delta)\leq \min(x^\alpha)\rbrace.
$$ Given a finite set $M$ of terms, its \emph{Pommaret span} is $\cup_{x^\alpha\in M} \mathcal C_{\mathcal P}(x^\alpha)$.
The finite set of terms $M$ is a \emph{Pommaret basis $\mathcal P((M))$} of the ideal $(M)$ if the Pommaret span of $M$ coincides with the set of terms in $(M)$ and the Pommaret cones of the terms in $M$ are pairwise disjoint. 
\end{definition}

\begin{theorem}\label{qsPB}\citep[Definition 4.3 and Proposition 4.4]{Seiler2009II} \ 
A monomial ideal $J$ is quasi-stable if and only if it has a Pommaret basis $\mathcal P(J)$. Further, the ideal $J$ is stable if and only if $\mathcal P(J)$ coincides with $B_J$.
\end{theorem}

In the present paper, we will mainly consider strongly stable ideals. In characteristic $0$, Borel-fixed ideals are strongly stable (see for example \citep{BS87}). Although when $J$ is strongly stable we have $\mathcal P(J)=B_J$, in the following we will use the notation of Pommaret bases, according to papers \citep{CMR13, BCR}

Recall that a {\em marked polynomial} is a polynomial $F$ together with a specified term of $\supp(F)$ that will be called {\em head term of $F$} and denoted by $\Ht(F)$ (see \citep{RS}).  

\begin{definition}\label{def:cose marcate} \citep[Definition 5.1]{CMR13}
Let $J\subset S$ be a quasi-stable ideal.

A {\em $\mathcal P(J)$-marked set} (or marked set over $\mathcal P(J)$) $G$ is a set of homogeneous monic marked polynomials $F_\alpha$ in $S$ such that  the head terms $Ht(F_\alpha)=x^\alpha$ are pairwise different and form the Pommaret basis $\mathcal P(J)$ of $J$, and $\supp(F_\alpha-x^\alpha)\subset \mathcal N(J)$.

A $\mathcal P(J)$\emph{-marked basis} (or marked basis over $\mathcal P(J)$) $G$ is a $\mathcal P(J)$-marked set such that $\mathcal N(J)$ is a basis of $S/(G)$ as a $K$-module, i.e.~$S=(G)\oplus \langle \cN (J) \rangle$ as a $K$-module.
\end{definition} 

Let $\mathfrak j\subset R$ be a quasi-stable ideal and $m$ a non-negative integer. Setting $J:={\mathfrak j}\cdot S$, we now recall the affine counterpart of Definition \ref{def:cose marcate}. 

\begin{definition}\label{marked affini} \citep[Definition 4.1]{BCR}
Let $\mathfrak j\subset R$ be a quasi-stable ideal and $m$ a non-negative integer.

A {\em $[\mathcal P({\mathfrak j}),m]$-marked set} $\mathfrak G$ is a set of monic marked  polynomials $f_\alpha$ of $R$ such that the head terms $\Ht(f_\alpha)=x^\alpha$ are pairwise different and form the Pommaret basis $\mathcal P({\mathfrak j})$ of $\id j$, and $\supp(f_\alpha-x^\alpha)\subseteq \mathcal N(\id j)_{\leq t}$ with $t=\max\{m,\vert \alpha \vert\}$. 

A $[\mathcal P({\mathfrak j}),m]$-marked set $\mathfrak G=\{f_\alpha\}_{x^\alpha \in \mathcal P({\mathfrak j})}$ is a {\em $[\mathcal P({\mathfrak j}),m]$-marked basis} if there is a $\mathcal P({J_{\geq m}})$-marked basis $G$ such that for every $x^\alpha \in \mathcal P({\mathfrak j})$ the term $x_0^{k_\alpha}f_\alpha^h$ belongs to $G$ for a suitable integer $k_\alpha$.

The {\em $[\mathcal P({\mathfrak j}),m]$-marked family} is the set of all the ideals $I\subseteq R$ that are generated by a $[\mathcal P({\mathfrak j}),m]$-marked basis.
\end{definition}

\begin{lemma}\label{lemma:fondamentale} 
\citep[Lemma 6.1(i) and Definition 6.2]{BCR}
An ideal $I\subset R$ belongs to a $[\mathcal P({\mathfrak j}),m]$-marked family  if and only if $R_{\leq t}= I_{\leq t}\bigoplus \langle \mathcal N(\mathfrak j)_{\leq t}\rangle$, for every $t\geq m$. 
\end{lemma}

\begin{definition}\label{relazione affine}\citep[Definition 4.2]{BCR}
Let $\id G=\{f_\alpha\}_{x^\alpha\in\mathcal P(\id j)}$ be a $[\mathcal P(\id j),m]$-marked set. We denote by $\id G_{\mathcal P}$ the set of polynomials
$$\id G_{\mathcal P}=\{x^\delta f_\alpha  \ \vert f_\alpha \in \id G \text{ and } \ x^\delta x^\alpha \in \mathcal C_{\mathcal P}(x^\alpha) \}\subseteq (\id G)$$
that are marked on the terms of $\id j$ in the natural way $Ht(x^\delta f_\alpha )=x^\delta x^\alpha$. 
We will denote by $\xrightarrow{\ \id G_{\mathcal P}\ }$ the reflexive and transitive closure of the  following reduction relation on $R$: $f$ is in relation with $f'$ if $f'=f-\lambda x^\delta f_\alpha $, where $x^\delta f_\alpha \in \id G_{\mathcal P}$ and $\lambda\neq 0_A$ is the coefficient of the term $x^{\alpha+\delta}$ in $f$.
We will write $f \xrightarrow{\ \id G_{\mathcal P}\ }_\ast h$ if $f \xrightarrow{\ \id G_{\mathcal P}\ } h$ and $h \in \langle \mathcal N(\id j)\rangle$.
\end{definition}

The reduction procedure of Definition \ref{relazione affine} is Noetherian and confluent \citep[Proposition 4.3]{BCR}. This is due to the fact that the Pommaret cones of a Pommaret basis give a partition on the the set of the terms belonging to the quasi-stable ideal $\mathfrak j$.

Thanks to the features of the reduction relation $ \xrightarrow{\ \id G_{\mathcal P}\ }_\ast$, we are able to explicitly construct a polynomial ring $K[C]$ and an ideal $\id U\subset K[C]$ so that the affine scheme $\Mf_{\mathcal P({\mathfrak j}),m}:=\mathrm{Spec}(K[C]/\id U)$ parameterizes the $[\mathcal P({\mathfrak j}),m]$-marked family. We now sketch this construction, which is explicitly stated in \citep[Section 6 and Algorithm $\textsc{MarkedScheme}(\id j, m)$]{BCR}.

For every $x^\alpha \in \mathcal P(\id j)$, let $t_\alpha$ be the integer $\max\{m,\vert\alpha\vert\}$. We define the following set of parameters: 
\begin{equation*} 
C:=\{C_{\alpha\eta}\ \vert \ x^\alpha\in \mathcal P(\id j), x^\eta \in \mathcal N(\id j)_{\leq t_\alpha}\}.
\end{equation*}
For each $x^\alpha \in \mathcal P(\id j)$, we consider the following marked polynomial in $K[C][\mathbf x]$:
\[
f_\alpha:=x^\alpha-\sum_{x^\eta \in \mathcal N(\id j)_{\leq t_\alpha}}C_{\alpha\eta}x^\eta.
\]
We collect these marked polynomials in the following $[\mathcal P(\id j),m]$-marked set:
\begin{equation}\label{eq:MSwithParam}
\id G:=\{f_\alpha\ \vert\ x^\alpha \in \mathcal P(\id j)\}\subseteq K[C][\mathbf x].
\end{equation}
For every $x^\beta \in \id j_{\leq m}\setminus \mathcal P(\id j)$, let $h_\beta$ be the unique polynomial in $\langle \mathcal N(\id j)\rangle_{K[C]}$ such that $x^\beta\xrightarrow{\ \id G_{\mathcal P} \ }_\ast h_\beta$. We write 
\begin{equation*} 
h_\beta=h_{\beta}^{\left(\leq m\right)}+h_{\beta}^{(>m)}
\end{equation*}
with $h_{\beta}^{(\leq m)}\in \langle \mathcal N(\id j)_{\leq m}\rangle_{K[C]}$ and $h_{\beta}^{(>m)}\in \langle\mathcal N(\id j)_{\leq \deg(h_{\beta})}\setminus \mathcal N(\id j)_{\leq m}\rangle_{K[C]}$.

For every $f_\alpha \in \id G$ and for every $x_i>\min(x^\alpha)$, let $h_{i\alpha}$ be the unique polynomial in $\langle \mathcal N(\id j)\rangle\subseteq K[C][\mathbf x]$ such that $x_if_\alpha\xrightarrow{\ \id G_{\mathcal P} \ }_\ast h_{i\alpha}$.

We denote by $\id U_1\subset K[C]$ the ideal generated by the coefficients in $K[C]$ of the terms $x^{\eta}$ in the polynomials $h_{\beta}^{(>m)}$, for every $x^\beta \in \id j_{\leq m}\setminus \mathcal P(\id j)$, and we denote by $\mathfrak U_2$ the ideal generated by the coefficients of the terms $x^{\eta}$ in the polynomials $h_{i\alpha}$, for every $f_\alpha \in \id G$ and for every $x_i>\min(x^\alpha)$.

The scheme $\Mf_{\mathcal P({\mathfrak j}),m}=\mathrm{Spec}(K[C]/\id U_1+\id U_2)$ is called {\em $[\mathcal P({\mathfrak j}),m]$-marked scheme}. 

\begin{theorem}\label{th:cose marcate} \citep[Theorem 6.6 and Propositions 6.12 and 5.6]{BCR}
A $[\mathcal P({\mathfrak j}),m]$-marked family is parameterized by the locally closed subscheme $\Mf_{\mathcal P({\mathfrak j}),m}$ of the Hilbert scheme $\hilbp$, where $p(t)$ is the affine Hilbert polynomial of $R/\mathfrak j$. 

If $\rho$ is the satiety of $\mathfrak j$ and $m\geq \rho-1$, then $\Mf_{\mathcal P({\mathfrak j}),m}$ is an open subscheme of $\hilbp$.
If $\cN(\id j)$ is finite, for every $m$ the ideal $\id U_1$ is $(0)$ and $\mathrm{Spec}(K[C]/\id U_1+\id U_2)=\mathrm{Spec}(K[C]/\id U_2)$. 
\end{theorem}

In next statement, we denote by $X$ a point of $\hilbp$ and assume that the $[\mathcal P({\mathfrak j}),m]$-marked scheme $\Mf_{\mathcal P({\mathfrak j}),m}$ contains $X$ up to a suitable change of coordinates, where $m\geq \rho-1$ and $\rho$ is the satiety of $\mathfrak j$.

\begin{corollary}\label{cor:tangente}
The Zariski tangent space to $\hilbp$ at $X$ is equal to the Zariski tangent space to $\Mf_{\mathcal P({\mathfrak j}),m}$ at $X$ and it can be explicitly computed by marked bases techniques.
\end{corollary}

\begin{proof} It is enough to observe that $\Mf_{\mathcal P({\mathfrak j}),m}$ is an open subscheme of $\hilbp$ thanks to Theorem \ref{th:cose marcate} and that the Zariski tangent space to $\Mf_{\mathcal P({\mathfrak j}),m}$ at $X$ can be computed by the generators of the ideal $\id U_1+\id U_2$ of Theorem \ref{th:cose marcate}.
\end{proof}
\begin{remark}\label{rem:tangente}
Concerning an effective computation of the Zariski tangent space to a  
marked scheme $\Mf_{\mathcal P({\mathfrak j}),m}$ at the origin  
$\mathfrak j$, we can use the same techniques that are described in  
\citep[Sections 3 and 4]{LR} for the so-called Gr\"obner strata. If we  
want to compute the Zariski tangent space at another point, we perform  
the change of coordinates that brings this point in the origin, as  
usual.  This procedure is computationally effective, while other  
approaches are much more tricky (see for instance \cite[Sections 4 and  
6]{Hu}).

Furthermore, along the paper we often need to certify that the  
dimension of the Zariski tangent space to a Hilbert scheme at one of  
its points is not reliant on the characteristic of the field. In order  
to do this we proceed in the following way.

By Corollary \ref{cor:tangente} we explicitly compute the Zariski  
tangent space to $\hilb_{p(t)}^n$ at a point $\id a$  
of $\Mf_{\mathcal P({\mathfrak j}),m}$ 
by the Jacobi criterium on the ideal defining $\Mf_{\mathcal  
P({\mathfrak j}),m}$, for  suitable $\id j$ and $m$.  
We make this first computation assuming  
$\mathrm{char}(K)=0$. We call $\id Z$ a  
set of independent linear polynomials that define the  
tangent space to $\hilb_{p(t)}^n$ at the point $\id a$. We assume that  
the linear polynomials in $\id Z$ are reduced w.r.t.~a suitable term  
order. In order to check that for every characteristic of the field  
the polynomials in $\id Z$ stay independent, we  
consider the matrix $M_{\id Z}$ containing the coefficients of the  
variables $c_{i,j}$ which are head terms of the polynomials in $\id  
Z$. If the determinant of $M_{\id Z}$ is divisible by a prime integer  
$p$, then we check that the polynomials in $\id Z$ are independent  
also when $\mathrm{char}(K)=p>0$.

We can further improve this procedure when we are interested in  
computing the dimension of a component of $\hilb_{p(t)}^n$ that  
contains a family $\mathcal F$ in $\Mf_{\mathcal P({\mathfrak j}),m}$ and we already know the dimension of $\mathcal F$.
 In this case, we  
take a generic point $\mathfrak a$ in $\mathcal F$ and compute the  
dimension of the Zariski tangent space to $\hilb_{p(t)}^n$ at  
$\mathfrak a$, as before. If the dimension of $\mathcal F$ is equal to  
that of the Zariski tangent space, it is evident that the point  
corresponding to $\mathcal F$ is smooth in $\hilb_{p(t)}^n$ and that  
the dimension of $\mathcal F$ is also the dimension of the irreducible  
component containing $\mathcal F$.
If the factorization of the determinant of $M_{\id Z}$ takes too much  
time, we provide the following improvement. We choose another point  
$\id a_2$ belonging to $\mathcal F$. We compute the Zariski tangent  
space to $\hilb_{p(t)}^n$ at the point $\id a_2$ and denote by $\id  
Z_2$ a set of reduced linear independent polynomials that define the  
tangent space to $\hilb_{p(t)}^n$ at the point $\id a_2$ and by  
$M_{\id Z_2}$ the matrix containing the coefficients of the variables  
$c_{i,j}$ which are head terms of the polynomials in $\id Z_2$. We  
then just need to check that the polynomials of $M_{\id Z}$ (or of  
$M_{\id Z_2}$) stay independent in characteristic $p$, for every $p$  
dividing the greatest common divisor of the determinant of $M_{\id Z}$  
and of the determinant of $M_{\id Z_2}$.
\end{remark}

We end this section with the following result that is inferred from \citep{FR} and is analogous to results contained in \citep{LR} for the homogeneous case. We first need to describe an adjustment to the affine case of the notion of segment (for details on segments see \citep{CLMR}).

\begin{definition}\label{segmento affine}
Let $\id j$ be a strongly stable ideal in $R$, $m$ a positive integer. The ideal $\id j$ is an {\em affine $m$-segment} if there is a weight vector $\omega \in \mathbb N^n$ such that for every $x^\alpha \in B_{\id j}$, $\deg_\omega(x^\alpha)>\deg_\omega(x^\gamma)$ for every $x^\gamma\in \cN(\id j)_{\leq t}$, with $t=\max\{m,\vert\alpha\vert\}$.
\end{definition}

\begin{theorem}\label{th:connessione}
If $\id j\subset R$ is an affine $m$-segment, then every irreducible component $\mathcal M$ of $\Mf_{\mathcal P({\mathfrak j}),m}$ contains  $\id j$, hence $\Mf_{\mathcal P({\mathfrak j}),m}$ is a connected scheme. 
If moreover the point corresponding to $\id j$ is smooth on $\mathcal M$, then $\mathcal M$ is isomorphic to an affine space.
\end{theorem}

\begin{proof}
Let $\omega \in \mathbb N^n$ be a weight vector with respect to whom  $\mathfrak j$ is an affine $m$-segment. Then, $\Mf(\id j,m)$ is a $\omega$-cone with vertex in $\mathfrak j$ by \citep[Corollary 2.7]{FR} and the thesis follows.  
\end{proof}

\section{Gorenstein points and double-generic initial ideals}\label{sec:gore}

In this section, we highlight a relation between the locus of Gorenstein points defined by graded Gorenstein $K$-algebras with Hilbert function of type $(1,n,n,1)$ and the notion of double-generic initial ideal.  {\em Here, we consider any infinite field $K$ with characteristic different from $2$ and $3$}.

A Gorenstein scheme $X$  
is a scheme such that the structure sheaf
in every point $x\in X$ is Gorenstein. Recall that the locus of points in $\hilbp$ representing Gorenstein schemes is an open subset \citep{Stoia,GM}. 

We will consider zero-dimensional Gorenstein schemes in an open neighborhood, which can be studied by our affine techniques \citep{BCR}. Hence, following \citep[Definition 2.1]{IK99} and \citep{BH}, we now recall some main notions and already known results for Artin $K$-algebras. 

Let $A$ be a local Artin $K$-algebra and $M$ its maximal ideal. The {\em socle} of $A$ is the annihilator $Soc(A):=(0:_A M)=\{h\in A \ \vert \ hM=0\}$.  Then, $A$ is called {\em Gorenstein} if $\dim_K Soc(A)=1$. An Artin $K$-algebra is {\em Gorenstein} if its localization at every maximal ideal is a Gorenstein (local) $K$-algebra. The {\em socle degree} of a graded Artin Gorenstein $K$-algebra $A$ is the maximum degree $j$ such that $A_j\neq 0$. 

The following result is due to Macaulay, as observed in \citep{IK99}.

\begin{lemma}\citep[Definition 1.11 and Lemma 2.12]{IK99} \label{lemma:apolarity}
There is a bijection between the polynomials of degree $j$ in the ring of divided power series, up to scaling, and the set of graded Artin Gorenstein quotient rings of $R$ of socle degree $j$. This correspondence associates to a form $F$ of degree $j$ the quotient $A_F:=R/Ann(F)$, where $Ann(F)$ is computed by apolarity.
\end{lemma}

The divided powers can be identified with polynomials either in characteristic 0, or when the degree is less than the characteristic. Hence, in our hypothesis on the characteristic of $K$, we can state the following result in terms of hypersurfaces.

\begin{theorem}\label{th:irreducible} 
\citep[Theorem 3.31]{IE1978} The set of cubic hypersurfaces, which determine all the graded Artin Gorenstein $K$-algebras with Hilbert function $(1,n,n,1)$ by apolarity, is a non-empty open  
irreducible subset of $\mathbb P(R_3)$.  
\end{theorem}

Let $\mathrm{Gor}(1,n,n,1)$ denote the subset of the projective space consisting of the hypersurfaces $F$ such that the Hilbert function of $A_F$ is $(1,n,n,1)$. 
By Theorem \ref{th:irreducible}, $\mathrm{Gor}(1,n,n,1)$ corresponds to an irreducible locally closed subset in a Hilbert scheme and we denote by $\overline{\mathrm{Gor}}(1,n,n,1)$ its closure. The following definitions and results are crucial in our study of $\overline{\mathrm{Gor}}(1,n,n,1)$.

\begin{definition} \citep{BCR-glin}
An irreducible closed subset $Y$ of a Hilbert scheme is called a {\em GL-stable subset} if is invariant under the action of the general linear group.
\end{definition}

Every GL-stable subset $Y$ of a Hilbert scheme contains at least one point corresponding to a strongly stable ideal. Given a term order, among the strongly stable ideals that define points of $Y$, we can find a special strongly stable ideal which is the saturation of the generic initial ideal of the generic (and general) point of $Y$  \citep[Proposition 4(b)]{BCR-glin}. 

\begin{definition} \citep[Definition 5]{BCR-glin}
The saturation of the generic initial ideal of the generic (and general) point of a GL-stable subset $Y$  is called the {\em double-generic initial ideal of $Y$} and is denoted by $G_Y$.
\end{definition}

The notion of double-generic initial ideal has been introduced and investigated for the first time in \citep{BCR-glin}, also in the more general setting of Grassmannian, with the terminology of extensors. 

\begin{proposition}\label{prop:GLstable}
$\overline{\mathrm{Gor}}(1,n,n,1)\subseteq \hilb^n_{2n+2}$ is a GL-stable subset, in particular it has a double-generic initial ideal.
\end{proposition}

\begin{proof}
Theorem \ref{th:irreducible} implies that the closure $\overline{\mathrm{Gor}}(1,n,n,1)$ in $\hilb^n_{2n+2}$ is a closed irreducible subset, hence it is GL-stable because it is also invariant under the action of the general linear group, by construction.
\end{proof}

Next result contains some of the main properties of a double-generic initial ideal. First, we need to recall the following definition. 

\begin{definition}\label{def:ordinamento} \citep[Definition 6]{BCR-glin}
Let $J$ and $H$ be monomial ideals in $S$ such that $S/J$ and $S/H$ have the same Hilbert polynomial $p(t)$. Let $r$ be the Gotzmann number of $p(t)$ and consider the set of generators $B_{J_{\geq r}}=\{\tau_1,\dots,\tau_q\}$ and $B_{I_{\geq r}}=\{\sigma_1,\dots,\sigma_q\}$ ordered by a term order $>$, where $q=\binom{n+r}{n}-p(r)$. We write $J >\!\!> H$ if $\tau_i \geq \sigma_i$ for every $i\in\{1,\dots,q\}$.
\end{definition}

\begin{lemma}\label{lemma:dgii}\citep[Propositions 2 and 3, Definition 5, Theorems 3 and 4, Remark 5]{BCR} Let $>$ be a term order, $Y$ a GL-stable subset of $\hilbp$ and $G_Y$ its double-generic initial ideal. 
\begin{itemize}
\item[(i)] For every ideal $I$ defining a point of $Y$, $\mathrm{gin}(I)$ and $\mathrm{in}(I)$ define points of $Y$.
\item[(ii)] There exists the maximum among all the Borel-fixed ideals defining points of $Y$ with respect to the partial order $>\!\!>$, and this maximum is $G_Y$.
\item[(iii)] There is a non-empty open subset $V$ of $Y$ such that $\mathrm{gin}(I)=\mathrm{in}(I)=G_Y$ for every saturated ideal $I$ defining a point in $V$. 
\end{itemize}
\end{lemma}

\section{Local Gorenstein $K$-algebras with Hilbert function $(1,7,7,1)$ define smoothable points} \label{GorSmooth}

In this section, we consider the Hilbert scheme $\hilb^7_{16}$ parameterizing zero-dimensional subschemes of $\PP^7_K$ of length $16$, {\em over an infinite field $K$ with characteristic different from $2$ and $3$}. Recall that, up to a generic change of coordinates, we can identify every point of $\hilb^7_{16}$ with an ideal in $R=K[x_1,\dots,x_7]$, not necessarily homogeneous. Hence, we consider the polynomial ring $R$ and the ideals in $R$ with affine Hilbert polynomial $p(t)=16$. A double-generic initial ideal $J$ will be considered in its {\em affine} version too, that is its dehomogenization $\mathfrak j:=J^a=B_J\cdot R$.

The lex-point of $\hilb^7_{16}$ corresponds to the following lex-segment ideal in $R$:
\[
\id j_{\mathrm{lex}}= (x_7,x_6,x_5,x_4,x_3,x_2,x_1^{16}).
\]
It is well-known that $\id j_{\mathrm{lex}}$ is a smooth point of the smoothable component $\mathcal R_{16}^7$ of dimension $7\cdot 16=112$, because the general point of $\mathcal R_{16}^7$ is a reduced scheme of $16$ distinct points.

We can compute the complete list of $561$ strongly stable ideals of $R$ lying in $\hilb^7_{16}$ by the algorithm described in \citep{CLMR} (and further developed and implemented in \citep{L} and generalized for quasi-stable ideals and Borel-fixed ideals in positive characteristic in \citep{Ber}). Among them, we focus on the following one:
\begin{align*}
\id j_G = &(x_7^2, x_7x_6, x_7x_5, x_7x_4, x_7x_3, x_7x_2, x_7x_1, x_6^2, x_6x_5, x_6x_4, x_6x_3, x_6x_2, x_6x_1, x_5^2, x_5x_4, x_5x_3, \\
&x_5x_2, x_5x_1, x_4^2, x_4x_3, x_4x_2, x_4x_1^2, x_3^3, x_3^2x_2, x_3^2x_1, x_3x_2^2, x_3x_2x_1, x_3x_1^2, x_2^3, x_2^2x_1, x_2x_1^2, x_1^4).
\end{align*}
By the constructive tools of \citep{BCR} and by theoretical results on the double-generic initial ideal, we now show that $\id j_G$ is the generic initial ideal w.r.t.~lex term order of a general ideal defining a graded (Artin) Gorenstein $K$-algebra with Hilbert function $(1,7,7,1)$.

\begin{theorem}\label{th:j_G}
The strongly stable ideal $\mathfrak j_G$ is the double-generic initial ideal of the GL-stable subset $\overline{\mathrm{Gor}}(1,7,7,1)$ w.r.t.~the lex order. In particular, it is the generic initial ideal w.r.t.~the lex order of a general point of $\mathrm{Gor}{(1,7,7,1)}$.
\end{theorem}

\begin{proof}All the computations we refer to can be repeated using the polynomials and values stored in the ancillary files available  in the folder ``Theorem \ref{th:j_G}''.

Thanks to the already cited correspondence with cubic hypersurfaces (see Lemma \ref{lemma:apolarity}), by apolarity we now explicitly construct a random ideal defining a graded Artin Gorenstein $K$-algebra with Hilbert function $(1,7,7,1)$ and having $\mathfrak j_G$ as initial ideal with respect to lex order. By \citep[Theorem 3]{BCR-glin} we can observe that $\mathfrak j_G$ is the generic initial ideal of $\mathrm{Ann}(F)$, because  $\mathfrak j_G$ is the maximum w.r.t.~the order of Definition \ref{def:ordinamento} among all the strongly stable ideals with Hilbert function $(1,7,7,1)$ (see \citep[Theorem 15.18]{Ei}). 
By \citep[Theorem 3]{BCR-glin}, this fact also implies that $\mathfrak j_G$ is the double-generic initial ideal of $\overline{\mathrm{Gor}}(1,7,7,1)$. By \citep[Proposition 4]{BCR-glin} we can then conclude that $\mathfrak j_G$ is the generic initial ideal w.r.t.~the lex order of a general point of $\mathrm{Gor}(1,7,7,1)$.

We randomly choose the following cubic form $F$ in $K[x_1,\dots,x_7]$
$$\begin{array}{l}
F = 2\,x_{1}^{3}-3\,x_{1}^{2}x_{2}-6\,x_{1}^{2}x_{4}-6\,x_{1}^{2}x_{5}-3\,x_{1}^{2}x_{7}+9\,x_{1}x_{2}^{2}+12\,x_{1}x_{2}x_{3}+12\,x_{2}x_{1}x_{4}+\\
+12\,x_{2}x_{1}x_{5}+12\,x_{2}x_{1}x_{6}+6\,x_{2}x_{1}x_{7}+6\,x_{1}x_{3}^{2}+6\,x_{1}x_{3}x_{4}+6\,x_{1}x_{3}x_{5}+12\,x_{1}x_{3}x_{6}+\\
+6\,x_{1}x_{3}x_{7}+6\,x_{1}x_{4}^{2}+12\,
x_{1}x_{4}x_{5}+6\,x_{1}x_{4}x_{6}+6\,x_{1}x_{4}x_{7}+6\,x_{1}x_{5}^{2}+6\,x_{1}x_{5}x_{6}+\\
+6\,x_{1}x_{5}x_{7}+6\,x_{1}x_{6}^{2}+6\,x_{1}x_{6}x_{7}+3\,x_{1}x_{7}^{2}-x_{2}^{3}+3\,x_{2}^{2}x_{3}-9\,x_{2}^{2}x_{4}-6\,x_{2}^{2}x_{5}+\\
-3\,x_{2}^{2}x_{6}-6\,x_{2}^{2}x_{7}+3\,x_{2}x_{3}^{2}-12\,x_{2}x_{3}x_{4}-6\,x_{2}x_{3}
x_{5}-6\,x_{2}x_{3}x_{6}-12\,x_{2}x_{3}x_{7}-3\,x_{2}x_{4}^{2}+\\
-12\,x_{2}x_{4}x_{5}-6\,x_{2}x_{5}^{2}-6\,x_{2}x_{5}x_{6}-6\,x_{2}x_{5}x_{7}+3\,x_{2}x_{6}^{2}-6\,x_{3}^{2}x_{4}-3\,x_{3}^{2}x_{5}-3\,x_{3}^{2}x_{6}+\\
-6\,x_{3}^{2}x_{7}-6\,x_{3}x_{4}x_{5}-3\,x_{3}x_{5}^{2}-6\,x_{3}x_{5}x_{6}-6\,x_{3}x_{5}x_{7}+3\,x_{3}x_{6}^{2}-5\,x_{4}^{3}-6\,x_{4}^{2}x_{5}+\\
-6\,x_{4}^{2}x_{6}-3\,x_{4}^{2}x_{7}-6\,x_{4}x_{5}^{2}-6\,x_{4}x_{5}x_{6}-6\,x_{4}x_{5}x_{7}
-6\,x_{4}x_{6}^{2}-12\,x_{4}x_{6}x_{7}-9\,x_{4}x_{7}^{2}+\\
-3\,x_{5}^{3}-3\,x_{5}^{2}x_{6}-3\,x_{5}^{2}x_{7}-3\,x_{5}x_{6}^{2}-6\,x_{5}x_{6}x_{7}-3\,x_{5}x_{7}^{2}-2\,x_{6}^{3}-6\,x_{6}^{2}x_{7}+\\
-6\,x_{6}x_{7}^{2}-2\,x_{7}^{3}.
\end{array}$$
Let $\mathrm{Ann}(F)\subset K[x_1,\dots,x_7]$ be the ideal that we obtain by apolarity from $F$. We check that $A_F:=K[x_1,\dots,x_7]/\mathrm{Ann}(F)$ is a graded Gorenstein $K$-algebra with Hilbert function $(1,7,7,1)$. We can also observe that $A_F$ is local as we expected, because $A_F$ is Artin and graded. The reduced Gr\"obner basis w.r.t.~the lex order of the ideal $\mathrm{Ann}(F)$ is given by the following $32$ polynomials (in bold the initial term of each polynomial):

\noindent$\mathfrak{f}_{1}=\mathbf{x_7^2} - 4x_1x_4 - 2x_3^2 + x_2x_3 - 2x_1x_3 - x_2^2 + 4x_1x_2$,\\
$\mathfrak{f}_{2}=\mathbf{x_6x_7} - x_1x_4 - x_2x_3 + x_1x_2$,\\
$\mathfrak{f}_{3}=\mathbf{x_5x_7} + x_1x_4 + x_1x_3 - x_1x_2$,\\
$\mathfrak{f}_{4}=\mathbf{x_4x_7} + 2x_1x_4 + 2x_3^2 - 2x_2x_3 + 5x_1x_3 - 5x_1x_2 - x_1^2$,\\
$\mathfrak{f}_{5}=\mathbf{x_3x_7} + 3x_1x_4 + x_2x_3 + 2x_1x_3 - 3x_1x_2$,\\
$\mathfrak{f}_{6}=\mathbf{x_2x_7} + 3x_1x_4 + x_2x_3 + 2x_1x_3 - 3x_1x_2$,\\
$\mathfrak{f}_{7}=\mathbf{x_1x_7} - x_1x_4 - x_1x_3 + x_1x_2$,\\
$\mathfrak{f}_{8}=\mathbf{x_6^2} + x_1x_4 + x_3^2 - 3x_2x_3 - x_1x_3 + x_2^2 + x_1x_2 + x_1^2$,\\
$\mathfrak{f}_{9}=\mathbf{x_5x_6} + x_1x_4 + x_1x_3 - x_1x_2$,\\
$\mathfrak{f}_{10}=\mathbf{x_4x_6} - x_1x_4 - x_2x_3 + x_1x_2$,\\
$\mathfrak{f}_{11}=\mathbf{x_3x_6} + 4x_1x_4 + x_2x_3 + 2x_1x_3 - 4x_1x_2$,\\
$\mathfrak{f}_{12}=\mathbf{x_2x_6} + 4x_1x_4 + x_2x_3 + 2x_1x_3 - 4x_1x_2$,\\
$\mathfrak{f}_{13}=\mathbf{x_1x_6} - x_1x_3$,\\
$\mathfrak{f}_{14}=\mathbf{x_5^2} + 2x_1x_4 + 2x_3^2 - 3x_2x_3 + 2x_2^2 - x_1^2$,\\
$\mathfrak{f}_{15}=\mathbf{x_4x_5} + x_1x_4$,\\
$\mathfrak{f}_{16}=\mathbf{x_3x_5} + x_1x_4 + x_1x_3 - x_1x_2$,\\
$\mathfrak{f}_{17}=\mathbf{x_2x_5} + x_1x_4$,\\
$\mathfrak{f}_{18}=\mathbf{x_1x_5} - x_1x_4$,\\
$\mathfrak{f}_{19}=\mathbf{x_4^2} - 7x_1x_4 - 4x_3^2 + 2x_2x_3 - 8x_1x_3 - x_2^2 + 11x_1x_2 + x_1^2$,\\
$\mathfrak{f}_{20}=\mathbf{x_3x_4} + 3x_1x_4 + x_2x_3 + 2x_1x_3 - 3x_1x_2$,\\
$\mathfrak{f}_{21}=\mathbf{x_2x_4} + 3x_1x_4 + x_2x_3 + x_1x_3 - 2x_1x_2$,\\
$\mathfrak{f}_{22}=\mathbf{x_1^2x_4} + x_1^3$,\\
$\mathfrak{f}_{23}=\mathbf{x_3^3}, \quad  \mathfrak{f}_{24}=\mathbf{x_2x_3^2} - \frac{1}{2}x_1^3,\quad  \mathfrak{f}_{25}=\mathbf{x_1x_3^2} - x_1^3, \quad \mathfrak{f}_{26}=\mathbf{x_2^2x_3} - \frac{1}{2}x_1^3$,\\
$\mathfrak{f}_{27}=\mathbf{x_1x_2x_3} - x_1^3, \quad \mathfrak{f}_{28}=\mathbf{x_1^2x_3}, \quad \mathfrak{f}_{29}=\mathbf{x_2^3} + \frac{1}{2}x_1^3, \quad \mathfrak{f}_{30}=\mathbf{x_1x_2^2} - \frac{3}{2}x_1^3,$\\
$\mathfrak{f}_{31}=\mathbf{x_1^2 x_2}+ \frac{1}{2}x_1^3, \quad \mathfrak{f}_{32}=\mathbf{x_1^4}$.

\noindent Then, the monomial ideal $\mathfrak j_G$ is the initial ideal of $\mathrm{Ann}(F)$ w.r.t.~lex order.
\end{proof}

\begin{remark}
The ideal $\mathfrak j_G$ belongs to $\overline{\mathrm{Gor}}(1,7,7,1)\setminus {\mathrm{Gor}}(1,7,7,1)$, because is the initial ideal of an ideal in ${\mathrm{Gor}}(1,7,7,1)$ but does not define a Gorenstein scheme.
\end{remark}

\begin{proposition}
$\Mf_{\mathcal P({\mathfrak j_G}),3} \cap \mathrm{Gor}{(1,7,7,1)} \neq \emptyset$.
\end{proposition}

\begin{proof}
From the proof of Theorem \ref{th:j_G} we see that the $32$ polynomials $\mathfrak f_i$, which generate the ideal $\mathrm{Ann}(F)$, form a Gr\"obner basis which is also a $[\mathcal P({\mathfrak j_G}),3]$-marked basis. Thus, $\mathrm{Ann}(F)$ belongs to the family of ideals defining the $[\mathcal P({\mathfrak j_G}),3]$-marked scheme $\Mf_{\mathcal P({\mathfrak j_G}),3}$.   
\end{proof}

With a suitable choice of values for the parameters occurring in the defining ideal $\mathfrak U_2$ of $\Mf_{\mathcal P({\mathfrak j_G}),3}$, we obtain that the following polynomials form a $[\mathcal P({\mathfrak j_G}),3]$-marked basis $\mathfrak G_{\tau}$, for every $\tau \in \mathbb A^1_K$:\\
$F_1=
\mathfrak{f}_1-9\,\tau x_{{7}}+16\,\tau x_{{4}}+ \frac{11}{2}\,\tau x_{{3}}-7\,\tau x_{{2}}+ 2\,\tau x_{{1}}-8\,{\tau }^{2}$,\\
$F_{{2}}=
\mathfrak{f}_2-\frac{1}{2}\,\tau x_{{6}}+\tau x_{{4}}+\frac{1}{2}\,\tau x_{{3}}-\tau x_{{2}}+\frac{1}{2}\,\tau x_{{1}}-\frac{1}{2}\,{\tau }^{2}$,\\
$F_{{3}}=
\mathfrak{f}_3-\tau x_{{7}}-\frac{1}{2}\,\tau x_{{5}}-\tau x_{{4}}-\tau x_{{3}}+\tau x_{{2}}-\frac{1}{2}\,\tau x_{{1}}+{\tau }^{2}$,\\
$F_{{4}}=
\mathfrak{f}_5+11\,\tau x_{{7}}-{\frac {45\,\tau}{2}}x_4-8\,\tau x_{{3}}+ 5\,\tau x_{{2}}-\frac{3}{2}\,\tau x_{{1}}+17\,{\tau }^{2}$,\\
$F_{{5}}=
\mathfrak{f}_5-3\,\tau x_{{4}}-3\,\tau x_{{3}}+3\,\tau x_{{2}}-\frac{3}{2}\,\tau x_{{1}}+\frac{3}{2}\,{\tau }^{2}$,\\
$F_{{6}}=
\mathfrak{f}_6-\frac{1}{2}\,\tau x_{{7}}-3\,\tau x_{{4}}-\frac{5}{2}\,\tau x_{{3}}+ \frac{5}{2}\,\tau x_{{2}}-\frac{3}{2}\,\tau x_{{1}}+\frac{7}{4}\,{\tau }^{2}$,\\
$F_{{7}}=
\mathfrak{f}_7-\tau x_{{2}}+\tau x_{{3}}+\tau x_{{4}}-\tau x_{{7}}$,\\
$F_{{8}}=
\mathfrak{f}_8+2\,\tau x_{{6}}-\tau x_{{4}}+\frac{1}{2}\,\tau x_{{3}}+2\,\tau x_{{2}}+\frac{1}{2}\,\tau x_{{1}}-{\frac {13\,{\tau }^{2}}{4}}$,\\
$F_{{9}}=
\mathfrak{f}_9-\tau x_{{6}}-\tau x_{{4}}-\tau x_{{3}}+\tau x_{{2}}-\frac{1}{2}\,\tau x_{{1}}+\frac{1}{2}\,{\tau }^{2}$,\\
$F_{{10}}=
\mathfrak{f}_{10}-\tau x_{{6}}+\tau x_{{4}}+\frac{1}{2}\,\tau x_{{3}}-\tau x_{{2}}+\frac{1}{2}\,\tau x_{{1}}-\frac{1}{2}\,{\tau }^{2}$,\\
$F_{{11}}=
\mathfrak{f}_{11}-4\,\tau x_{{4}}-\frac{5}{2}\,\tau x_{{3}}+4\,\tau x_{{2}}-2\,\tau x_{{1}}+2\,{\tau }^{2}$,\\
$F_{{12}}=
\mathfrak{f}_{12}-\frac{1}{2}\,\tau x_{{6}}-4\,\tau x_{{4}}-\frac{5}{2}\,\tau x_{{3}}+4\,\tau x_{{2}}-2\,\tau x_{{1}}+2\,{\tau }^{2}$,\\
$F_{{13}}=
\mathfrak{f}_{13}+\tau x_{{3}}-\tau x_{{6}}-x_{{3}}x_{{1}}$,\\
$F_{{14}}=
\mathfrak{f}_{14}-14\,\tau x_{{5}}-2\,\tau x_{{4}}-
\frac{5}{2}\,\tau x_{{3}}+ 6\,\tau x_{{2}}-4\,\tau x_{{1}}+{\frac {29\,{\tau }^{2}}{2}}$,\\
$F_{{15}}=
\mathfrak{f}_{15}+2\,{\tau }^{2}-\tau x_{{1}}-2\,\tau x_{{4}}-\tau x_{{5}}$,\\
$F_{{16}}=
\mathfrak{f}_{16}-\tau x_{{4}}-2\,\tau x_{{3}}+\tau x_{{2}}-\frac{1}{2}\,\tau x_{{1}}+\frac{1}{2}\,{\tau }^{2}$,\\
$F_{{17}}=
\mathfrak{f}_{17}-\frac{1}{2}\,\tau x_{{5}}-\tau x_{{4}}-\tau x_{{2}}-\tau x_{{1}}+\frac{3}{2}\,{\tau }^{2}$,\\
$F_{{18}}=
\mathfrak{f}_{18}+\tau x_{{4}}-\tau x_{{5}}$,\\
$F_{{19}}=
\mathfrak{f}_{19}-20\,\tau x_{{7}}+ 37\,\tau x_{{4}}+15\,\tau x_{{3}}-14\,\tau x_{{2}}+\frac{7}{2}\,\tau x_{
{1}}-{\frac {95\,{\tau }^{2}}{4}}$,\\
$F_{{20}}=
\mathfrak{f}_{20}-3\,\tau x_{{4}}-\frac{7}{2}\,\tau x_{{3}}+3\,\tau x_{{2}}-\frac{3}{2}\,\tau x_{{1}}+\frac{3}{2}\,{\tau }^{2}$,\\
$F_{{21}}=
\mathfrak{f}_{21}-\frac{7}{2}\,\tau x_{{4}}-\frac{3}{2}\,\tau x_{{3}}+\tau x_{{2}}-2\,\tau x
_{{1}}+\frac{5}{2}\,{\tau }^{2}$,\\
$F_{{22}}=
\mathfrak{f}_{22}+4\,{\tau }^{3}-5\,{\tau }^{2}x_{{1}}-4\,{\tau }^{2}x_{{2}}+{\tau }^{2}x_{{4}}+4\,\tau x_{{1}}x_{{2}}-2\,\tau x_{{1}}x_{{4}}$,\\
$F_{{23}}=
\mathfrak{f}_{23}-2\,{\tau }^{3}+2\,{\tau }^{2}x_{{1}}-4\,{\tau }^{2}x_{{2}}-3\,{\tau }^{2}x_{{3}}+4\,{\tau }^{2}x_{{4}}+4\,\tau x_{{1}}x_{{2}}-4\,\tau x_{{1}}x_{{4}}+6\,\tau x_{{2}}x_{{3}}-2\,\tau {x_{{3}}}^{2}$,\\
$F_{{24}}=
 \mathfrak{f}_{24}-\frac{1}{2}\,\tau {x_{{3}}}^{2}+4\,\tau x_{{2}}x_{{3}}-6\,\tau x_{{1}}x_{{4}}-6\,\tau x_{{3}}x_{{1}}+8\,\tau x_{{1}}x_{{2}}-\frac{1}{2}\,\tau {x_{{1}}}^{2}+ 6\,{\tau }^{2}x_{{4}}+\\
 +4\,{\tau }^{2}x_{{3}}-8\,{\tau }^{2}x_{{2}}+\frac{9}{2}\,{\tau }^{2}x_{{1}}-\frac{7}{2}\,{\tau }^{3}$,\\
$F_{{25}}= 
 \mathfrak{f}_{25}-5\,{\tau }^{3}+7\,{\tau }^{2}x_{{1}}-12\,{\tau }^{2}x_{{2}}+8\,{\tau }^{2}x_{{3}}+8\,{\tau }^{2}x_{{4}}-\tau {x_{{1}}}^{2}+ 12\,\tau x_{{1}}x_{{2}}-8\,\tau x_{{3}}x_{{1}}+\\
 -8\,\tau x_{{1}}x_{{4}}-\tau {x_{{3}}}^{2}$,\\
$F_{{26}}=
\mathfrak{f}_{26}+3\,\tau x_{{2}}x_{{3}}-6\,\tau x_{{1}}x_{{4}}-6\,\tau x_{{3}}x_{{1}}+8\,\tau x_{{1}}x_{{2}}-\frac{1}{2}\,\tau {x_{{1}}}^{2}+ 6\,{\tau }^{2}x_{{4}}+{\frac {17\,{\tau }^{2}x_{{3}}}{4}}+\\
 -8\,{\tau }^{2}x_{{2}}+\frac{9}{2}\,{\tau }^{2}x_{{1}}-\frac{7}{2}\,{\tau }^{3}$,\\
$F_{{27}}=
 \mathfrak{f}_{27}-\tau x_{{2}}x_{{3}}-8\,\tau x_{{1}}x_{{4}}-\frac{17}{2}\,\tau x_{{3}}x_{{1}}+12\,\tau x_{{1}}x_{{2}}-\tau {x_{{1}}}^{2}+8\,{\tau }^{2}x_{{4}}+ \frac{17}{2}\,{\tau }^{2}x_{{3}}+\\
 -12\,{\tau }^{2}x_{{2}}+7\,{\tau }^{2}x_{{1}}-5\,{\tau }^{3}$,\\
$F_{{28}}=
\mathfrak{f}_{28}-6\,{\tau }^{3}+6\,{\tau }^{2}x_{{1}}-12\,{\tau }^{2}x_{{2}}+5\,{\tau }^{2}x_{{3}}+12\,{\tau }^{2}x_{{4}}+12\,\tau x_{{1}}x_{{2}}-6\,\tau x_{{3}}x_{{1}}-12\,\tau x_{{1}}x_{{4}}$,\\
$F_{{29}}=
 \mathfrak{f}_{29}+\frac{5}{2}\,\tau {x_{{2}}}^{2}-
2\,\tau x_{{1}}x_{{4}}-6\,\tau x_{{3}}x_{{1}}+12\,\tau x_{{1}}x_{{2}}+1/2\,\tau {x_{{1}}}^{2}+2\,{\tau }^{2}x_{{4}}+ 6\,{\tau }^{2}x_{{3}}+\\
 -{\frac {61\,{\tau }^{2}x_{{2}}}{4}}-\frac{13}{2}\,{\tau }^{2}x_{{1}}+{\frac {51\,{\tau }^{3}}{8}}$,\\
$F_{{30}}=
 \mathfrak{f}_{30}-\tau {x_{{2}}}^{2}-10\,\tau x_{{1}}x_{{4}}-6\,\tau x_{{3}}x_{{1}}+7\,\tau x_{{1}}x_{{2}}-\frac{3}{2}\,\tau {x_{{1}}}^{2}+10\,{\tau }^{2}x_{{4}}+ 6\,{\tau }^{2}x_{{3}}+\\
 -7\,{\tau }^{2}x_{{2}}+{\frac {55\,{\tau }^{2}x_{{1}}}{4}}-{\frac {43\,{\tau }^{3}}{4}}$,\\
$F_{{31}}=
\mathfrak{f}_{31}-10\,\tau x_{{1}}x
_{{4}}-6\,\tau x_{{3}}x_{{1}}+14\,\tau x_{{1}}x_{{2}}+10\,{\tau }^{2}x_{{4}}+6\,{\tau}^{2}x_{{3}}-15\,{\tau }^{2}x_{{2}}+ \frac{1}{2}\,{\tau }^{2}x_{{1}}-{\tau }^{3}$,\\
$F_{{32}}=
\mathfrak{f}_{32}+45\,{\tau }^{4}-40\,{\tau }^{3}x_{{1}}+32\,{\tau }^{3}x_{{2}}-64\,{\tau }^{3}x_{{4}}-6\,{\tau }^{2}{x_{{1}}}^{2}-32\,{\tau }^{2}x_{{1}}x_{{2}}+64\,{\tau }^{2}x_{{1}}x_{{4}}$.

\noindent From now, for every $\tau\in \mathbb A^1_K$ we denote by $\id i_{\tau}\subset K[x_1,\dots,x_7]$ the ideal generated by the $[\mathcal P({\mathfrak j_G}),3]$-marked basis $\id G_{\tau}=\{F_1,\dots,F_{32}\}$. Note that for $\tau=0$ we obtain the ideal $\id i_0=\mathrm{Ann}(F)$, which defines a scheme with support in a single point, as we have already observed in the proof of Theorem \ref{th:j_G}. For every $\tau\in \mathbb A^1_K\setminus\{0\}$ we have a different situation because every ideal $\id i_{\tau}$ defines a scheme whose support contains at least the following $8$ distinct affine points:
\begin{align*}
&\left(\tau,\displaystyle\displaystyle\frac{1}{2} \tau,2\tau,\tau,\tau,0,\displaystyle\displaystyle\frac{1}{2}\tau\right),
\quad \left(\tau,\displaystyle\frac{-7}{2}\tau,0,\tau,\tau,0,\displaystyle\frac{1}{2}\tau\right),
\quad \left(-3\tau,\displaystyle\frac{1}{2}\tau,0,\tau,\tau,0,\displaystyle\frac{1}{2}\tau\right),\\
&\left(\tau,\displaystyle\frac{1}{2}\tau,0,\tau,\tau,0,\displaystyle\frac{1}{2}\tau\right),
\quad\left(-7\tau,\displaystyle\frac{-15}{2}\tau,-8\tau,\tau,\tau,-8\tau,\displaystyle\frac{1}{2}\tau\right),
\quad\left(\tau,\displaystyle\frac{1}{2}\tau,0,\tau,\tau,-2\tau,\displaystyle\frac{1}{2}\tau\right),\\
&\left(\tau,\displaystyle\frac{1}{2}\tau,0,\tau,13\tau,0,\displaystyle\frac{1}{2}\tau\right),
\quad\left(-7\tau,\displaystyle\frac{17}{2}\tau,0,9\tau,9\tau,0,\displaystyle\frac{1}{2}\tau\right).
\end{align*}
This observation is crucial for next result.

\begin{lemma}\label{lemma:piattezza7}
There is a flat family of ideals which is contained in $\mathrm{Mf}_{\mathcal P({\mathfrak j_G}),3}\cap \mathcal R_{16}^7$ such that the special fiber corresponds to a Gorenstein point defined by a graded $K$-algebra with Hilbert function $(1,7,7,1)$.
\end{lemma}

\begin{proof} We prove that the family of ideals $\{\id i_{\tau}\}_{\tau}$, which are generated by the $[\mathcal P({\mathfrak j_G}),3]$-marked bases $\id G_{\tau}$, is contained in the smoothable component $\mathcal R_{16}^7$. All the computations we refer to can be repeated using the polynomials and values stored in the ancillary files available  in the folder ``Lemma \ref{lemma:piattezza7}''.

By construction, the ideals $\id i_{\tau}$ belong to the marked scheme $\Mf_{\mathcal P({\mathfrak j_G}),3}$ which embeds as an open subset in the punctual Hilbert scheme $\hilb_{16}^7$. So, these ideals define a flat family over $\mathbb A^1$. As we have already recalled in Section \ref{sec:gore}, the locus of points in a Hilbert scheme representing all the Gorenstein schemes is an open subset. Hence, the intersection of this locus with the flat family $\{\id i_{\tau}\}_{\tau}$, which is non-empty because $\id i_0$ represents a Gorenstein point, is an open subset of the family. Thus, we find at least a value $\overline\tau\neq 0$ such that $\id i_{\overline\tau}$ represents a Gorenstein point.

By computational tools, we have already found that for every $\tau\neq 0$ the ideal $\id i_{\tau}$ defines a scheme whose support contains at least $8$ distinct points and hence components of multiplicity at most $9$. So, the ideal $\id i_{\overline\tau}$ is smoothable due to the fact that for $d\leq 9$ the locus of Gorenstein points is a Hilbert scheme of $d$ points is irreducible (see \citep[Theorem A]{CN9} together with \citep[Proposition 1.2]{JJB2017}). Thus, every ideal $\id i_\tau$ belongs to the smoothable component $\mathcal R_{16}^7$ because the family is irreducible. In particular, the special fiber $\id i_0$ belongs to $\mathcal R_{16}^7$, because $\mathcal R_{16}^7$ is closed and irreducible.
\end{proof}

\begin{remark}\label{rem:fuori schema}
Although the ideal $\id i_0$ corresponds to a Gorenstein point in $\mathrm{Gor}(1,7,7,1)$, for every $\tau\neq 0$ the ideal $\id i_{\tau}$ corresponds to a point which does not belong to $\mathrm{Gor}(1,7,7,1)$. 
Indeed, we constructed the family of ideals $\{\id i_{\tau}\}_{\tau}$ by letting the term $x_7$ have a non-null coefficient in the polynomial $F_{19}$. The term $x_7$ is higher than the head term $x_4^2$ of $F_{19}$ with respect to lex term order. This fact implies that the initial ideal of $\id i_{\tau}$ is not $\id j_G$ and is not comparable with $\id j_G$ w.r.t.~the order of Definition \ref{def:ordinamento} (see \citep[Theorem 3 and Proposition 8]{BCR-glin}). Thus, the generic initial ideal of $\id i_{\tau}$ is different from $\id j_G$. Recalling that $\id j_G$ is the double-generic initial ideal of $\overline{\mathrm{Gor}}(1,7,7,1)$, we obtain our claim. We can also observe that the initial ideal of $\id i_{\tau}$ is {\em closer} to the lex-segment ideal than $\id j_G$ w.r.t.~the order of Definition \ref{def:ordinamento}.
\end{remark}

\begin{lemma}\label{lemma:liscezza}
There is a smooth Gorenstein point defined by a graded $K$-algebra with Hilbert function $(1,7,7,1)$ belonging to the smoothable component $\mathcal R_{16}^7$. 
\end{lemma}

\begin{proof}All the computations we refer to can be repeated using the polynomials and values stored in the ancillary files available  in the folder ``Lemma \ref{lemma:liscezza}''.

We prove that the ideal $\id i_0$ defines a smooth Gorenstein point in the smoothable component $\mathcal R_{16}^7$. We already know that  $\id i_0$ defines a Gorenstein point with Hilbert function $(1,7,7,1)$. Moreover, by Lemma \ref{lemma:piattezza7} and by construction, the ideal $\id i_0$ represents a point that belongs to $\mathcal R_{16}^7\cap \Mf_{\mathcal P({\mathfrak j_G}),3}$. Due to Corollary \ref{cor:tangente}, we can compute the Zariski tangent space to $\hilb^7_{16}$ at $\id i_0$ from the polynomials of the $[\mathcal P({\mathfrak j_G}),3]$-marked basis of $\id i_0$, obtaining that the dimension of the Zariski tangent space to $\hilb^7_{16}$ at the point $\id i_0$ is $112=16 \times 7$, i.e.~the dimension of the smoothable component (see Section \ref{sec:marked}, in particular Remark \ref{rem:tangente}, for the method we apply in order to perform this computation).
\end{proof}

\begin{theorem} \label{Gor7}
Over an infinite field $K$ with $\mathrm{char}(K)\neq 2,3$, every Gorenstein point defined by a local $K$-algebra with Hilbert function $(1,7,7,1)$ is smoothable.
\end{theorem}

\begin{proof}
By Lemmas \ref{lemma:piattezza7} and \ref{lemma:liscezza}, there exists a smooth  Gorenstein point with Hilbert function $(1,7,7,1)$ in the smoothable component $\mathcal R_{16}^7$. These facts imply that also all the other Gorenstein points with the same Hilbert function belong to $\mathcal R_{16}^7$, i.e.~are smoothable, because the locus $\mathrm{Gor}{(1,n,n,1)}$ in $\hilbp$ of the schemes parameterizing homogeneous Gorenstein ideals with Hilbert function $(1,n,n,1)$ is irreducible (see Theorem \ref{th:irreducible} and Proposition \ref{prop:GLstable}). Finally, we conclude due to \citep[Proposition 1.1]{JJ2017}.
\end{proof}

\begin{remark} Recall that Theorem \ref{Gor7} covers the unique case not treated in the range considered in \citep[Lemma 6.21]{IK99} for a study of non-smoothable Gorenstein points.
\end{remark}

\section{Some irreducible components of $\hilb^7_{16}$}\label{IrrComp7}
\label{sec:components}

In this section 
we obtain interesting information about the components of $\hilb^7_{16}$ from a study of the irreducible components of $\Mf_{\mathcal P({\mathfrak j_G}),3}$  {\em over any infinite field $K$ with characteristic different from $2$ and $3$}. By construction, the marked scheme $\Mf_{\mathcal P({\mathfrak j_G}),3}$ is indeed the open subscheme of $\hilb^7_{16}$ where the Pl\"ucker coordinate corresponding to the monomial ideal ${\mathfrak j_G}^h$ is invertible. So, the closures of the components of $\Mf_{\mathcal P({\mathfrak j_G}),3}$ are irreducible components of $\hilb^7_{16}$.

There are three already known irreducible components in $\hilb^7_{16}$:
\begin{itemize}
\item the smoothable component $\mathcal R_{16}^7$, which contains schemes supported on 16 distinct points. This component has dimension $7\cdot 16=112$;
\item two more components that we denote by $\mathcal{IE}_1$ and $\mathcal{IE}_2$, whose existence is a consequence of the results by \cite{IE1978}, who exhibit an \emph{elementary component} of $\hilb^4_8$,  and of \cite[Theorem 1.1]{CEVV} in which an elementary component of $\hilb_8^7$ is described. 
\end{itemize}

In Section \ref{sec:IEcomp}, inspired by \cite{Hu}, we will briefly investigate the elementary component of $\hilb_8^7$.
In Section \ref{sec:newComp} we consider the components $\mathcal{IE}_1$ and $\mathcal{IE}_2$, showing that they have dimension $\leq 98$ and $\leq 105$, respectively, by a direct computation. Furthemore, we prove the existence of two more irreducible components, the first one of dimension $161$, the second one of dimension between $150$ and $153$.

\subsection{Iarrobino-Emsalem elementary component in $\hilb^7_{8}$}\label{sec:IEcomp}

As pointed out by the anonymous referee, \cite{IE1978} proved the existence of an elementary component $\mathcal{IE}$ in $\hilb^4_8$. Furthermore, \cite{CEVV} proved that  the analogous elementary component  of $\hilb^7_8$ has dimension $49$.

We now study this elementary component of $\hilb_8^7$ using marked bases. We recover that its dimension is $49$ finding an open subset that is isomorphic to an affine space of dimension $49$.

\begin{theorem}\label{thm:IEdim}
The elementary component $\mathcal{IE}$ of $\hilb_8^7$ is rational (and has dimension $49$).
\end{theorem}

\begin{proof}
All the computations we refer to can be repeated using the polynomials and values stored in the ancillary files available  in the folder ``Theorem \ref{thm:IEdim}''. 

We will study the marked scheme $\Mf_{\mathcal P({\mathfrak j}),2}$, where:
\[
\id j=(x_{{7}},x_{{6}},x_{{5}},x_{{4}}^{2},x_{{3}}x_{{4}},x_{{2}}x_{{4}},x
_{{1}}x_{{4}},x_{{3}}^{2},x_{{2}}x_{{3}},x_{{1}}x_{{3}},x_{{1}}^{3
},x_{{1}}^{2}x_{{2}},x_{{1}}x_{{2}}^{2},x_{{2}}^{3}).
\]
The ideal $\id j$ has affine Hilbert function $(1, 4, 3)$ and defines a $0$-dimensional scheme with support on the origin of $\mathbb A^7$.
The following terms form a monomial basis for $K[\mathbf x]/\id j$, hence are the terms appearing in the support of $[\mathcal P(\id j),2]$-marked bases:
\[
\{x_{{2}}^{2},x_{{2}}x_{{1}},x_{{1}}^{2},x_{{4}},x_{{3}},x_{{2}},x_
{{1}},1\}.
\]

We consider a $[\mathcal P(\id j),2]$-marked set $\id G$, like in \eqref{eq:MSwithParam}, and denote by $\overline{\id G}$ the $[\mathcal P(\id j),2]$-marked set that we obtain replacing $0$ to the parameters in the set  $\overline C:=C_1\cup C_2\cup C_3$, where:
\[C_1:=\{c_{i,8}\vert  i \in\{1,\dots, 10\}\},\quad
C_2:=\{c_{i,j}\vert i \in \{4,\dots, 10\},j\in\{4,\dots,7\}\},\]
\[C_3:=\{c_{i,j}\vert i\in \{11,\dots,14\},j\in\{1,\dots,8\}\}.\]
$C_1$ contains the constant terms (w.r.t.~to the $x_i$'s) of the marked polynomials in $\id G$; $C_2$ contains the coefficients of terms of degree 1 appearing in marked polynomials in $\id G$ with head term of degree $2$; $C_3$ contains all the parameters appearing in marked polynomials in $\id G$ with head terms of degree $3$.

Then, we denote by $\mathcal M_{\mathcal{IE}}$ the closed subscheme $\Mf_{\mathcal P({\mathfrak j}),2}\cap L$, where $L$ is the linear space defined by the vanishing of the parameters in $C_1\cup C_2\cup C_3$. Under this linear section, the ideal $\id U_1+\id U_2\subset K[C]$ that defines the scheme $\Mf_{\mathcal P({\mathfrak j}),2}$ reduces to $(0)$. Hence, $\mathcal M_{\mathcal{IE}}$ is isomorphic to an affine space of dimension $42$ and parameterizes the family of $[\mathcal P(\id j),m]$-marked bases $\overline{\id G}$. 
Furthermore, the generic point of $\overline{\id G}$  corresponds to a scheme having a structure of multiplicity 8 over the origin in $\mathbb A^7$, and the ideal defining this generic point has a set of generators (of degrees $<3$). Considering the possible translations of the origin, we obtain a family of ideals parameterized by a subscheme of $\Mf_{\mathcal P({\mathfrak j}),2}$ having dimension $42+7=49$.

By Corollary \ref{cor:tangente} (using the techiniques explained in Remark \ref{rem:tangente}), we compute that the dimension of the Zariski tangent space to $\hilb_8^7$ at the generic point of $\mathcal M_{\mathcal{IE}}$ is 49, for every characteristic different from 2 and 3. Hence, $\mathcal M_{\mathcal{IE}}$ is an open subset of the elementary component $\mathcal{IE}$ of $\hilb_8^7$, whose dimension is exactly 49 and which is isomorphic to an affine space.
\end{proof}

\begin{remark}
The family of ideals defined by the marked set $\overline{\id G}$ in the proof of Theorem \ref{thm:IEdim} has similar features to those of the families considered by \cite{Hu}.
\end{remark}

\begin{remark}
In $\hilb^4_8$ there is an elementary component of dimension 25 \citep{IE1978};  one might expect that an elementary component of $\hilb_8^7$ has dimension 40, because $\PP^4_K$ can be embedded in $\PP^7$ depending on $15$ free parameters.
Nevertheless, the dimension of this elementary component is higher than this naive estimate, because the general point of this component correspond  to a  scheme with structure of multiplicity 8 on one single point of $\mathbb P^{7}_K$ which is not contained in any $\mathbb P^4_K$. For instance, the point defined by the ideal proposed in the ancillary file for Theorem \ref{thm:IEdim} is not contained in any hyperplane of $\PP^7_K$.
\end{remark}

\subsection{Some new irreducible components of $\hilb^7_{16}$}\label{sec:newComp}

Statements and proofs of the present section are a byproduct of the computations carried on in Section \ref{GorSmooth}, in the following sense: while studying $\Mf_{\mathcal P({\mathfrak j_G}),3}$ in order to prove that there is a smooth Gorenstein point defined by a graded $K$-algebra with Hilbert function $(1,7,7,1)$ on the smoothable component, we see that the generators of the ideal defining the scheme $\Mf_{\mathcal P({\mathfrak j_G}),3}$  can reveal some properties of $\hilb^7_{16}$.

Our first result is a consequence of Theorem \ref{th:connessione}.  

\begin{proposition}\label{prop:componenti G}
The marked scheme $\Mf_{\mathcal P({\mathfrak j_G}),3}$ is a connected open subset of $\hilb^7_{16}$ and every irreducible component $\mathcal M$ of of $\Mf_{\mathcal P({\mathfrak j_G}),3}$ contains $\id j_G$. Moreover, if  $\id j_G$ is a smooth point for $\mathcal M$, then $\mathcal M$ is isomorphic to an affine space. 
\end{proposition}

\begin{proof}
The marked scheme $\Mf_{\mathcal P({\mathfrak j_G}),3}$ is an open subset of $\hilb^7_{16}$ due to \citep[Proposition 6.12(ii)]{BCR-glin}.
Furthermore, the ideal $\id j_G$ is an affine $3$-segment with respect to the weight vector $\omega=[11, 10, 9, 8, 6, 5, 4]$. We conclude by Theorem \ref{th:connessione}.
\end{proof}

Let $\mathcal M_1:=\Mf_{\mathcal P({\mathfrak j_G}),3}\cap \mathcal R_{16}^7$ denote the irreducible component of $\Mf_{\mathcal P({\mathfrak j_G}),3}$ that is obtained by intersecting $\Mf_{\mathcal P({\mathfrak j_G}),3}$ with the smoothable component $\mathcal R_{16}^7$. Thus, the dimension of $\mathcal M_1$ is $112$. We now highlight the existence of  other components  of $\Mf_{\mathcal P({\mathfrak j_G}),3}$. The result of Proposition \ref{prop:componenti G} suggests us to look for the irreducible components other that $\mathcal M_1$ containing the ideal $\id j_G$.

By the techniques described in \citep{BCR} and briefly recalled in Section \ref{sec:marked}, we obtain $\Mf_{\mathcal P({\mathfrak j_G}),3}$ as a closed subscheme of the affine space $\mathbb A^{512}$   defined by an ideal $\id U$ generated by $2160$ polynomials of degrees $d=3,4,5$ in the polynomial ring $K[C]$ in the $512=32\cdot 16$ variables $c_{i,j}$ with $i=1,\dots, 32$ ($32$ is the number of monomials in the monomial basis of ${\mathfrak j_G}$) and $j=1, \dots, 16$ ($16$ is the number of monomials outside $({\mathfrak j_G}$)).  

Recalling the construction of a $[\mathcal P({\mathfrak j_G}),3]$-marked set, we consider the terms outside $\id j_G$ of degree up to $3$ in the following order:
$$x_{1}^{3},x_{3}^{2},x_{3}x_{2},x_{2}^{2},x_{4}x_{1
},x_{3}x_{1},x_{2}x_{1},x_{1}^{2},x_{7},x_{6},x_{5}
,x_{4},x_{3},x_{2},x_{1},1.
$$

The computation of a primary decomposition of the ideal $\id U$ is unaffordable with Gr\"obner bases techniques. Thus, we look for other strategies that are partly  based on geometrical and/or algebraic observations, and partly on experimental calculations.

First, we observe that there are two more components of $\Mf_{\mathcal P({\mathfrak j_G}),3}$, that we call \emph{Iarrobino-Emsalem components} and are obtained from the elementary component $\mathcal{IE}$ of $\hilb_8^7$.

\begin{theorem}\label{prop:IEperjG}
There are two irreducible components of the marked scheme $\Mf_{\mathcal P({\mathfrak j_G}),3}$ whose dimensions are respectively $98$ and $105$.   
\end{theorem}

\begin{proof}
All the computations we refer to can be repeated using the polynomials and values stored in the ancillary files available  in the folder ``Theorem \ref{prop:IEperjG}''. 

Let $\mathcal M_{\mathcal{IE}_1}$ (resp.~$\mathcal M_{\mathcal{IE}_2}$) be the subset of $\Mf_{\mathcal P({\mathfrak j_G}),3}$ whose points correspond to the union of a scheme supported on the origin with multiplicity 8 and a scheme supported at another point with multiplicity $8$ (resp.~the union of a scheme supported on the origin with multiplicity 8 and a scheme supported on $8$ points). By Theorem \ref{thm:IEdim}, we can describe $\mathcal M_{\mathcal{IE}_1}$ (resp.~$\mathcal M_{\mathcal{IE}_2}$) by using $49+49=98$ (resp.~$49+8\cdot 7=105$) parameters, hence it is locally closed in   $\Mf_{\mathcal P({\mathfrak j_G}),3}$.

For what concerns  $\mathcal M_{\mathcal{IE}_1}$, it is equal to the image of the obvious map 
\[
\varphi_1: \mathcal{IE}\times \mathcal{IE}\setminus \mathcal Z_1\longrightarrow\mathcal M_{\mathcal{IE}_1} \subseteq\hilb_{16}^{7},
\]
where $\mathcal{IE}$ is the elementary component of $\hilb^7_8$ and  $\mathcal Z_1$ is the set containing couples of schemes supported on the same point. Note that $\varphi_1$ is a $2:1$ (hence, having $0$-dimensional fibers), so that the dimension of $\mathcal M_{\mathcal{IE}_1}$ is $\geq 98$.

For what concerns $\mathcal M_{\mathcal{IE}_2}$, we consider   the image of the obvious map 
\[
\varphi_2: \mathcal{IE}\times \mathcal L\setminus \mathcal Z_2\longrightarrow\mathcal M_{\mathcal{IE}_2} \subseteq\hilb_{16}^{7},
\]
where $\mathcal{IE}$ is the elementary component of $\hilb^7_8$, $\mathcal L$ is the open subset of smoothable schemes in $\hilb^7_8$, and  $\mathcal Z_2$ is the set containing couples of schemes supported on less than 9 points. The map $\varphi_2$ is a $1:1$ map and the dimension of $\mathcal M_{\mathcal{IE}_2}$ is $\geq 105$.

It is now sufficient to consider the ideal corresponding to a point of $\mathcal M_{\mathcal{IE}_1}$ (resp.~$\mathcal M_{\mathcal{IE}_2}$),  giving random values to the parameters describing  $\mathcal M_{\mathcal{IE}_1}$ (resp.~$\mathcal M_{\mathcal{IE}_2}$), thanks to Theorem \ref{thm:IEdim}.
We consider the term order given by  the weight vector $\omega=[11, 10, 9, 8, 6, 5, 4]$, as in the proof of Proposition \ref{prop:componenti G}, with ties  broken by reverse lexicographic order.
The Gr\"obner basis of this ideal, w.r.t.~the chosen weighted-degree order, 
has initial ideal equal to $\id j_G$, and the polynomials in it have degree not higher than $3$. Hence, the Gr\"obner basis of this ideal is a $[\mathcal P(\id j_G),3]$-marked set. Finally, one directly checks that the coefficients of the polynomials in this Gr\"obner basis satisfy the equations of the ideal $\id U_2$ that defines $\Mf_{\mathcal P({\mathfrak j_G}),3}$. 

Furthermore, thanks to Corollary \ref{cor:tangente}, a direct computation of tangent spaces gives that the dimension of the Zariski tangent space to $\hilb_{16}^7$ at a generic point of $\mathcal M_{\mathcal{IE}_1}$ (resp.~of $\mathcal M_{\mathcal{IE}_2}$) is exactly $98$ (resp.~$105$).  Hence $\dim \mathcal M_{\mathcal{IE}_1}\leq 98$ (resp. $\dim \mathcal M_{\mathcal{IE}_1}\leq 105$). We then obtain equality in both cases.
\end{proof}

\begin{remark}
In the proof of Theorem 	\ref{prop:IEperjG}, we can deduce that $\dim \mathcal M_{\mathcal{IE}_1}\leq 98$ (resp. $\dim \mathcal M_{\mathcal{IE}_1}\leq 105$) also arguing in the following way: if $\dim \mathcal M_{\mathcal{IE}_1}> 98$ (resp. $\dim \mathcal M_{\mathcal{IE}_2}>105$), this would imply that a scheme corresponding to any point of the elementary component  $\mathcal{IE}\subset \hilb^7_8$ could be deformed to a scheme supported on at least two distinct points.
\end{remark}

\begin{corollary}
The Hilbert scheme $\hilb_{16}^7$ has two irreducible components $\mathcal{IE}_1$ and $\mathcal{IE}_2$ different from $\mathcal R_{16}^7$: $\mathcal{IE}_1$ has dimension $\leq 98$ and its general point is the union of a scheme supported on a single point with multiplicity 8 and a scheme supported at another point with multiplicity $8$; $\mathcal{IE}_2$ has dimension $\leq 105$ and its general point is the union of a scheme supported on a single point with multiplicity 8 and a scheme supported on $8$ points.
\end{corollary}

Next Theorems \ref{th:componente2} and \ref{th:componente3} reveal the existence of other two componets of $\Mf_{\mathcal P({\mathfrak j_G}),3}$ thanks to further developments of the computations that have been performed in Section \ref{GorSmooth}.

\begin{theorem}\label{th:componente2}
There is an irreducible component $\widetilde{\mathcal M_2}$ of the marked scheme $\Mf_{\mathcal P({\mathfrak j_G}),3}$ that is rational, has dimension $161$, and its closure is an irreducible component of  $\hilb_{16}^7$.
\end{theorem}

\begin{proof} 
In the following, first we constructively find a rational subscheme $\mathcal M_2$ in $\Mf_{\mathcal P({\mathfrak j_G}),3}$ which is a complete intersection   of dimension $161$ in $\mathbb A^{512}$. Then, we consider an ideal $\mathfrak a$ corresponding to a point in $\mathcal M_2$ and, thanks to Corollary \ref{cor:tangente}, we compute the Zariski tangent space to $\hilb_{16}^7$ at the point $P_{\mathfrak a}$ corresponding to $\mathfrak a$ finding that it has dimension $161$. Thus, an irreducible component of $\hilb_{16}^7$ containing $P_{\mathfrak a}$  must have dimension at most $161$ and ${\mathcal M_2}$ exactly has dimension $161$. Hence, we will conclude that ${\mathcal M_2}$ is an irreducible component of $\Mf_{\mathcal P({\mathfrak j_G}),3}$. 
 
All the computations we refer to can be repeated using the polynomials and values stored in the ancillary files available  in the folder ``Theorem \ref{th:componente2}'. 

Let  $\id G$ be the $[\mathcal P(\id j_G),3]$-marked set as in \eqref{eq:MSwithParam}. We consider the ideal $\id U\subset K[C]$, that defines $\Mf_{\mathcal P({\mathfrak j_G}),3}$ as a subscheme  of $\mathbb A^{512}$, where $512=\vert C\vert$, according to Theorem \ref{th:cose marcate}.

We want to define a family of ideals generated by $[\mathcal P(\id j),m]$-marked bases whose homogenization defines a family of 0-dimensional schemes in $\mathbb P^7_K$ with a singularity in $[0:\dots:0:1]$. 
Consider the following subset of $C$:
\begin{equation}\label{eq:barC}
\overline C=\{c_{i,j}\,\vert\,j\geq 9\text{ or }i\geq 22\text{ and }j\geq 2\}.
\end{equation}
The parameters in $\overline C$ are the coefficients of the terms of degree lower than the degree of the corresponding head term, except for $i=32$. Let $L$ be the linear variety  defined by the vanishing of the parameters in $\overline C$.
We replace every $c_{i,j}\in \overline C$ with $0$ in $\id G$, obtaining a new $[\mathcal P(\id j_G),3]$-marked set that we denote by $\id G_L$. 

The remaining parameters are $179=8\cdot 21+ 11$ and have indexes either $i\leq 21$ and $j\leq 8$ or $i\geq 22$ and $j=1$.  

Intersecting the marked scheme $\Mf_{\mathcal P({\mathfrak j_G}),3}$ with the linear variety $L$, we obtain   $169$ equations in the variables $C\setminus \overline C$. Several of these equations are reducible, and one of the factors is $c_{32,1}$. We then focus on the subscheme $\mathcal M_2$  of  $\Mf_{\mathcal P({\mathfrak j_G}),3}\cap L$ defined by replacing each  
$f$ in the generators of $\id U_2$ by $f/c_{32,1},$ if $f$ is divisible by $c_{32,1}$. We denote by $\id U_{\mathcal M_2}$ this ideal in $K[C]/(\overline C)$.
 
Interreducing the polynomials in $\id U_{\mathcal M_2}$, we obtain $25$ polynomials $u_{i,j}$, which define a complete intersection of dimension $154$ in $K [C]/(\overline C)$ and allow the elimination of the following $25$ parameters in the generators of $\mathcal M_2$:
\[
c_{{1,8}},c_{{2,8}},c_{{3,8}},c_{{4,8}},c_{{5,8}},c_{{6,8}},c
_{{8,8}},c_{{9,8}},c_{{10,8}},c_{{11,8}},c_{{12,8}},c_{{14,8}},c_{{15,
8}},c_{{16,8}},c_{{17,8}},c_{{19,8}},c_{{20,8}},c_{{21,8}},c_{{23,1}},\]
\[
c_{{24,1}},c_{{25,1}},c_{{26,1}},c_{{27,1}},c_{{29,1}},c_{{30,1}}.
\]
Eliminating these $25$ paramaters from the equations defining $\mathcal M_2$ in $K[C]/(\overline C)$, we obtain $\id U_{\mathcal M_2}=(0)\subset K[C]/(\overline C)$. Hence, $\mathcal M_2$ is a complete intersection and has dimension $154$. 

We denote by $C_0$ the set of the remaining $154=179-25$ parameters, and keep on denoting by $\id G_L$ the marked basis $\mathfrak G$ after the elimination of the $25$ above parameters. The family of schemes defined by $\id G_L$ is then parameterized by $\mathcal M_2$, hence it depends on 154 free parameters and it is contained  in $\Mf_{\mathcal P({\mathfrak j_G}),3}$.  Thanks to the good behavior of  $[\mathcal P({\mathfrak j}),m]$-marked bases with respect to homogenization \citep[]{BCR}, and to  the fact that in the present hypotheses $\id U_1=(0)\subset K[C]$ (see by Theorem \ref{th:cose marcate}), the  homogenization with respect to $x_0$ of the ideal generated by $\id G_L$ defines a family of  $0$-dimensional schemes in $\mathbb P^7$ singular at $[0:\dots:0:1]$.  

Considering translations on the variables $x_1,\dots,x_7$, hence translations of the origin to another support point for the family of the schemes defined by the marked basis $\id G_L$, this family spreads to a larger family that is defined by a marked basis $\widetilde {\id G_L}$ which depends on $161=154+7$ parameters and which is still contained in $\Mf_{\mathcal P({\mathfrak j_G}),3}$. Denote by $\widetilde{\mathcal M_2}$ the subscheme of points corresponding to the ideals generated by $\widetilde {\id G_L}$. By construction, $\widetilde{\mathcal M_2}$ has dimension $161$. Moreover, it is a rational complete intersection because it depends on exactly $161$ parameters.

Now, we randomly choose particular values for the $154$ parameters in $C_0$ in order to obtain the $[\mathcal P({\mathfrak j_G}),3]$-marked basis of an ideal $\id a$ corresponding to a point of $\mathcal M_2\subset \widetilde{\mathcal M_2}$ (please, see the ancillary files  for a possible choice of the values for the $154$ parameters in $C_0$, the consequent values for the $25$ eliminable variables and the generators  
of the ideal $\id a$). 
 
Finally, by Corollary \ref{cor:tangente} (using the techiniques explained in Remark \ref{rem:tangente}) we explicitly compute  the Zariski tangent space at $\hilb_{16}^7$ at the point $\id a$ using the ideal defining $\Mf_{\mathcal P({\mathfrak j_G}),3}$ and obtain that its dimension is exactly $161$, as expected, for every $\mathrm{char}(K)\neq 2,3$.
\end{proof}

Simply comparing the dimensions we get the following result. 

\begin{corollary}\label{cor:dimM2}
The Hilbert scheme $\hilb_{16}^7$ has an irreducible component of dimension $161$, different from $\mathcal R_{16}^7$, $\mathcal{IE}_1$ and $\mathcal{IE}_2$.
\end{corollary}

As a byproduct of Theorem \ref{th:componente2}, we obtain the following result on $\hilb_{15}^7$, which is not known yet, up to our knowledge.

\begin{corollary}\label{cor:elementary component}
The Hilbert scheme $\hilb_{15}^7$ has an elementary component.
\end{corollary}

\begin{proof}
The scheme defined by the ideal $\id a$ belonging to $\mathcal M_2$, that we use in the proof of Theorem \ref{th:componente2}, is the union of a scheme supported on the origin with multiplicity 15,  and of   a simple point. This means that the general point of the irreducible component $\mathcal M_2\subset \hilb_{16}^7$ defines a scheme which is the union of a simple point and of a non-reduced structure of multiplicity $15$ on a different point. This implies the existence of an elementary component of $\hilb^7_{15}$.
\end{proof}

\begin{remark}\label{rm:supporto punto 15}
The ideal $\id a$ of Theorem \ref{th:componente2} defines a general point of $\mathcal M_2$ and is smooth in $\hilb_{16}^7$. 
Observe that the ideal $\mathfrak a$ is only one of the possible points we can consider in order to compute  the dimension of the Zariski tangent space to $\hilb_{16}^7$. We can also observe that the set $\overline C$ that is considered in the proof of Theorem \ref{th:componente2} generates the initial ideal  of the tangent space to $\hilb_{16}^7$ at the point corresponding to $\mathfrak a$ with respect to any term order.
\end{remark}

\begin{remark}
From the fact that $\id j_G$ is an affine $3$-segment with respect to the weight vector $\omega:=[11, 10, 9, 8, 6, 5, 4]$ we obtain that $\mathcal M_2$ is a cone, with vertex in $\id j_G$, with respect to a positive non-standard grading (see \citep[Corollary 2.7]{FR}). Thus, there is a projection in the Zariski tangent space to $\mathcal M_2$ at the origin which induces an isomorphism of $\mathcal M_2$ with its image (see \citep[Theorem 3.2]{FR}). This projection identifies a set of eliminable variables which is very useful, for example, in order to enhance the performance of the computations in this context.
\end{remark}

\begin{theorem}\label{th:componente3}
There is a subscheme $\widehat{\mathcal M_3}$ of $\Mf_{\mathcal P({\mathfrak j_G}),3}$ which is isomorphic to an affine space of dimension $150$, and every point of this closed subscheme corresponds to a 0-dimensional scheme supported on a single point. The irreducible component of  $\hilb^7_{16}$ containing $\widehat{\mathcal M_3}$ has dimension $\leq 153$.
\end{theorem}

\begin{proof} 
In the following, first we constructively find a subscheme $\widehat{\mathcal M}_3$ in $\Mf_{\mathcal P({\mathfrak j_G}),3}$ which is isomorphic to an affine space of dimension $150$. Then, we compute the Zariski tangent space to $\hilb_{16}^7$ at a point of $\widehat{\mathcal M}_3$ finding that it has dimension $153$. So an irreducible component of $\hilb_{16}^7$ containing $\widehat{\mathcal M}_3$ has dimension at most $153$. 
All the computations we refer to can be repeated using the polynomials and values stored in the ancillary files available  in the folder ``Theorem \ref{th:componente3}''.

As in the proof of Theorem \ref{th:componente2}, consider   the $[\mathcal P(\id j_G),3]$- marked set  $\id G$ as in \eqref{eq:MSwithParam} and the ideal $\id U\subset K[C]$ that defines $\Mf_{\mathcal P({\mathfrak j_G}),3}$ as a subscheme of $\mathbb A^{512}$, where $512=\vert C\vert$. 
We follow the same argument as in Theorem \ref{th:componente2}: 
again, we want to define a  family of 0-dimensional schemes in $\mathbb P^7_K$ with a singularity in $[0:\dots:0:1]$ whose defining ideals correspond to points in $\Mf_{\mathcal P({\mathfrak j_G}),3}$. 

Hence, we again consider a suitable set of parameters $\tilde C$, for which we refer to an ancillary file; this set is close to the one defined in  \eqref{eq:barC}, but there are a few variables  more and a few variables less:
\begin{equation}\label{eq:tildeC}
\widetilde C=\{c_{i,j}\,\vert\,j\geq 7\text{ or }i\geq 22\text{ and }1\leq j\leq 6\} \setminus (\widetilde{C_1} \cup  \widetilde{C_2} \cup\{ c_{{31,1}},c_{{31,2}},c_{{31,3}},c_{{31,4}},c_{{31,5}},c_{{31,6}}\} ),
\end{equation}
where
$$ \widetilde{C_1}= \left\{ c_{{7,7}},c_{{7,8}},c_{{13,7}},c_{{13,8}},c_{{18,7}}
 \right\}  \hbox{ and }
 \widetilde{C_2}=\left\{ c_{{1,7}},c_{{1,8}},c_{{2,7}},c_{{2,8}},c_{{3,7}},c_{{6,7}},c
_{{8,7}},c_{{8,8}},c_{{9,7}},c_{{12,7}}  \right\}.
$$
Let $\tilde L$ be the linear variety  defined by the vanishing of the parameters in $\tilde C$. Observe that the component $\mathcal M_2$ of Theorem \ref{th:componente2} does not contain $\tilde L$, because $c_{32,1}\in \widetilde C$. 

Intersecting $\Mf_{\mathcal P({\mathfrak j_G}),3}$ with $\tilde L$, the ideal $\id U_2$ reduces to the ideal $\id U_{\tilde L}$ generated by $45$ polynomials, of the type $f\cdot g$, with $f\in\{ 1, c_{{31,1}}, c_{{31,2}},c_{{31,3}},c_{{31,4}},c_{{31,5}}, c_{{31,6}}, c_{{31,1}}^2\}$ and    $g$ one of the following polynomials:\\
$  c_{{12,7}}-c_{{13,8}}   ,\quad 
   c_{{6,7}}-c_{{7,8}}   ,\quad
   -c_{{13,8}}c_{{18,7}}+c_{{9,7}}    , \quad 
   -2\,c_{{13,7}}c_{{13,8}}+c_{{8,7}}   , \quad    -c_{{7,8}}c_{{18,7}}+c_{{3,7}}  ,
$
 \\
$-2\,c_{{13,7}}c_{{13,8}}c_{{31,1}}+c_{{8,7}}c_{{31,1}}-c_{{13,8}}^{2}+c_{{8,8}}),\quad
   -c_{{7,7}}c_{{13,8}}-c_{{7,8}}c_{{13,7}}+c_{{2,7}}   ,\quad    -2\,c_{{7,7}}c_{{7,8}}+c_{{1,7}}   $
 \\
$-c_{{7,7}}c_{{13,8}}c_{{31,1}}-c_{{7,8}}c_{{13,7}}c_{{31,1}}+c_{{2,7}}
c_{{31,1}}-c_{{7,8}}c_{{13,8}}+c_{{2,8}}, \quad
-2\,c_{{7,7}}c_{{7,8}}c_{{31,1}}+c_{{1,7}}c_{{31,1}}-{c_{{7,8}}}^{2}+c
_{{1,8}}.
$
\vskip 1mm
Assuming that $c_{31,1}$ (or any variable $c_{31,j}$, $i=j, \dots, 6$) is generally non-zero on the locus we are studying,  we get then 10 polynomials  that allow us to write all the variables in $\tilde C_2$ as functions of those in  $\widetilde C_1$. 

We replace every $c_{i,j}\in  \widetilde C$ with $0$ and  those in $\widetilde C_2$ using the 10 relations above in the marked set $\id G$,  obtaining the $[\mathcal P(\id j_G),3]$-marked set $\mathcal F_3$ depending on 137  parameters. By construction 
  $\mathcal F_3$ is a marked basis.

We can apply a special class of changes of coordinates on the marked basis $\mathcal F_3$, still obtaining a family of schemes with the same features of those defined by $\mathcal F_3$. The changes of coordinates we use are listed in an ancillary file and they depend on $13$ free parameters $a_1, \dots, a_6, b_1, \dots, b_7$: for every value of the parameters $c_{i,j}$ and of the parameters $a_k, b_h$, we obtain a  point of $\Mf_{\mathcal P({\mathfrak j_G}),3}$.  The new  marked basis $\widehat{\mathcal F_3}$, which depends also on  the parameters $a_k, b_h$,  is free of dimension 150.

Note that the change of coordinates we consider are not as general as possible; however, performing  changes of coordinates depending on more than 13 parameters,  we  experimentally found  that the parameters are not free for the marked family, since a same marked basis  can be obtained by infinitely many different specializations of  the parameters.

Denote by $\widehat{\mathcal M_3}$ the subscheme of points corresponding to the ideals in $\widehat{\mathcal{F}_3}$. 
$\widehat{\mathcal M_3}$ is isomorphic to an affine space of dimension $150$. 
We randomly pick a point of $\widehat{\mathcal M_3}$, and by direct computation we see that the corresponding ideal generated by a marked basis in $\widehat {\mathcal F_3}$ is still a 0-dimensional scheme supported on a single point. By Corollary \ref{cor:tangente} (using the techiniques explained in Remark \ref{rem:tangente}), we compute that the dimension of the Zariski tangent space to the Hilbert scheme $\hilb_{16}^7$ at this point is $153$. Thus, the dimension of an irreducible component of 
 $\hilb^7_{16}$ containing $\widehat{\mathcal M_3}$ is $\leq 153$.
\end{proof}

\begin{remark}
We believe that the subscheme $\widehat{\mathcal M_3}$ can be further enlarged by using 3 more parameters, obtaining a subscheme of $\Mf_{\mathcal P({\mathfrak j_G}),3}$ of dimension $153$.  
\end{remark}

Analogously to Corollary \ref{cor:dimM2}, we get the following result by a comparison of dimensions. 

\begin{corollary}
The Hilbert scheme $\hilb_{16}^7$ has an irreducible component with dimension between $150$ and $153$, different from $\mathcal R_{16}^7$, $\mathcal{IE}_1$, $\mathcal{IE}_2$ and from the component containing $\widetilde{\mathcal M_2}$.
\end{corollary}

\begin{remark}
Note that the family of schemes defined by $\widetilde{\id G_L}$ of the proof of Theorem \ref{th:componente2} and that defined by  $\widehat{\mathcal F_3}$ of the proof of Theorem \ref{th:componente3} do not belong to $\mathrm{Gor}(1,7,7,1)$. Indeed, a general point of $\widetilde{\mathcal M_2}$ (resp.~a general point of $\widehat{\mathcal M}_3$) corresponds to a non-reduced structure over a point. Otherwise, the component of $\hilb_{16}^7$ containing $\widetilde{\mathcal M_2}$ (resp.~$\widehat{\mathcal M_3}$) would be a second component of $\hilb_{16}^7$ containing $\mathrm{Gor}(1,7,7,1)$, that is impossible for the irreducibility of $\mathrm{Gor}(1,7,7,1)$.
\end{remark}

\begin{corollary}\label{cor:tre componenti}
There are at least five irreducible components of $\hilb^7_{16}$ through the point corresponding to the ideal $\id j_G$.
\end{corollary}

\begin{proof} This is an immediate consequence of Propositions \ref{prop:componenti G} and \ref{prop:IEperjG} and Theorems \ref{th:componente2} and \ref{th:componente3}. 
\end{proof}

\section{Local Gorenstein $K$-algebras with Hilbert function $(1,5,5,1)$ define smoothable points}\label{GorSmooth2}

In this section we show that the same techniques of Section \ref{GorSmooth} allow to prove the already known smoothability of local Gorenstein $K$-algebras with Hilbert function $(1,5,5,1)$. We recall that different proofs of this case $(1,5,5,1)$ were presented in \citep{JJ}, contemporary to our first version given in \citep{BCR-arxiv}, and later in \citep{CN15} when $\mathrm{char}(K)\neq 2,3$. {\em Here, we test the results of our computations over a field of characteristic $0$}.

We consider the Hilbert scheme $\hilb^5_{12}$ parameterizing zero-dimensional subschemes of $\PP^5_K$ of length $12$. As before, we identify every point of $\hilb^5_{12}$ with an ideal in $R=K[x_1,\dots,x_5]$, non-necessarily homogeneous, with affine Hilbert polynomial $p(t)=12$.
The lex-point of $\hilb^5_{12}$ corresponds to the following lex-segment ideal in $R$:
\[
\mathfrak j_{\mathrm{lex}}= (x_5,x_4,x_3,x_2,x_1^{12}).
\]
It is well-known that $\mathfrak j_{\mathrm{lex}}$ is a smooth point of the smoothable component $\mathcal R_{16}^7$ of dimension $5\cdot 12=60$, because the general point of $\mathcal R_{16}^7$ is the reduced scheme of $12$ distinct points.

As in Section \ref{GorSmooth}, we compute the complete list of $92$ strongly stable ideals of $R$ lying on $\hilb^5_{12}$. Among them, we focus on the following one:
\small
\[
\mathfrak j_G=(x_{{5}}^{2},x_{{4}}x_{{5}},x_{{3}}x_{{5}},x_{{2}}x_{{5}},x_{{5}}x_{
{1}},x_{{4}}^{2},x_{{3}}x_{{4}},x_{{2}}x_{{4}},x_{{1}}x_{{4}},x_{{3
}}^{2},x_{{2}}^{2}x_{{3}},x_{{1}}x_{{2}}x_{{3}},x_{{1}}^{2}x_{{3}
},x_{{2}}^{3},x_{{1}}x_{{2}}^{2},x_{{1}}^{2}x_{{2}},x_{{1}}^{4
})
\]
\normalsize
By the constructive tools of \citep{BCR} and by theoretical results on the double-generic initial ideal, we now show that $\mathfrak j_G$ is the generic initial ideal w.r.t.~lex term order of a general ideal defining a graded Gorenstein $K$-algebra with Hilbert function $(1,5,5,1)$.

\begin{theorem}\label{th:J_G}
The ideal $\mathfrak j_G$ is the generic initial ideal w.r.t.~the lex term order of a general ideal defining a graded Gorenstein $K$-algebra with Hilbert function $(1,5,5,1)$.
\end{theorem}

\begin{proof}All the computations we refer to can be repeated using the polynomials and values stored in the ancillary files available  in the folder ``Theorem \ref{th:J_G}''.

We argue like in the proof of Theorem \ref{th:j_G}. 
Thanks to the already cited correspondence with cubic hypersurfaces (see Lemma \ref{lemma:apolarity}), by apolarity we now explicitly construct a random ideal defining a graded Artin Gorenstein $K$-algebra with Hilbert function $(1,5,5,1)$ and having $\mathfrak j_G$ as initial ideal with respect to lex order. By \citep[Theorem 3]{BCR-glin} we can observe that $\mathfrak j_G$ is the generic initial ideal of $\mathrm{Ann}(F)$, because  $\mathfrak j_G$ is the maximum w.r.t.~the order of Definition \ref{def:ordinamento} among all the strongly stable ideals with Hilbert function $(1,5,5,1)$ (see \citep[Theorem 15.18]{Ei}). 
By \citep[Theorem 3]{BCR-glin}, this fact also implies that $\mathfrak j_G$ is the double-generic initial ideal of $\overline{\mathrm{Gor}}(1,5,5,1)$. By \citep[Proposition 4]{BCR-glin} we can then conclude that $\mathfrak j_G$ is the generic initial ideal w.r.t.~the lex order of a general point of $\mathrm{Gor}(1,5,5,1)$.  

We randomly choose a cubic form $F$ in $K[x_1,\dots,x_5]$ and from $F$ compute the ideal $\mathrm{Ann}(F)\subset K[x_1,\dots,x_5]$ by apolarity. The reduced Gr\"obner basis w.r.t.~the lex order of the ideal $\mathrm{Ann}(F)$ is given by the following $17$ polynomials (in bold the initial term of each polynomial):

\noindent$\id f_1:=\mathbf{x_5^2}+4x_1^{2}+\frac {17}{3}x_1x_2-\frac {83}{12}  x_1 x_3-\frac {23}{4}\,x_2 x_3,$\\
$\id f_2:=\mathbf{x_4 x_5}-\frac{3}{4}\,x_2 x_3-\frac{5}{4}\,x_1 x_3+x_1 x_2,$\\ 
$\id f_3:=\mathbf{x_4^2 }+\frac {25}{6} x_2 x_3+x_2^{2}+{\frac {71}{18}} x_1 x_3- \frac {28}{9}x_1 x_2-5   x_1^{2},$\\
$\id f_4:=\mathbf{x_3 x_5}-\frac{3}{4}\,x_2 x_3+\frac{3}{4} x_1 x_3-x_1 x_2,$\\
$\id f_5:=\mathbf{x_3 x_4}-x_2x_3,$ \\
$\id f_6:=\mathbf{x_3^{2}}-{\frac {85}{24}}\,x_{{2}}x_3-{\frac {317}{72}}\,x_1 x_3+
\frac {71}{18}  x_1 x_2+2 x_1^{2},$\\
$\id f_7:=\mathbf{x_2x_5}-\frac{3}{4}\,x_{{2}}x_3-\frac{5}{4}\,x_{{1}}x_3+x_{{1}}x_2,$\\
$\id f_8:=\mathbf{x_2x_4}-x_{{2}}x_3-x_{{1}}x_3+x_{{1}}x_2,$\\
$\id f_9:=\mathbf{x_1x_5}-\frac{1}{4}x_{{2}}x_3+\frac{1}{4} x_1x_3-x_1 x_2,$\\
$\id f_{10}:=\mathbf{x_1x_4}-x_1 x_2$, $\id f_{11}:=\mathbf{x_2^2 x_3}+x_1^3$, $\id f_{12}:=\mathbf{x_2^3}+\frac{5}{9}x_1^3$, 
$\id f_{13}:=\mathbf{x_2x_1 x_3}-{\frac{11}{9}}x_1^3,$\\
$\id f_{14}:=\mathbf{x_1 x_2^{2}}-{\frac{8}{9}}x_1^3$, $\id f_{15}:=\mathbf{x_1^2 x_3}+ x_1^3$, $\id f_{16}:=\mathbf{x_1^2 x_2}+\frac{2}{3} x_1^3$, 
$\id f_{17}:=\mathbf{x_1^4}.$

\noindent We observe that $A_F:=K[x_1,\dots,x_5]/\mathrm{Ann}(F)$ is a graded Artin Gorenstein $K$-algebra with Hilbert function $(1,5,5,1)$, hence $A_F$ is local as we expected. By further computations, we obtain that $\mathfrak j_G$ is the initial ideal of $\mathrm{Ann}(F)$ w.r.t.~lex order, as claimed. 
\end{proof}

\begin{remark}
The strongly stable ideal $\mathfrak j_G$ is an affine $3$-segment with respect to the weight vector $\omega=[8, 7, 5, 4, 3]$. Hence, we can apply Theorem \ref{th:connessione} to $\mathfrak j_G$. Moreover, the ideal $\mathfrak j_G$ belongs to $\overline{\mathrm{Gor}}(1,5,5,1)\setminus {\mathrm{Gor}}(1,5,5,1)$, because is the initial ideal of an ideal in ${\mathrm{Gor}}(1,5,5,1)$ but does not define a Gorenstein scheme.
\end{remark}

The following straightforward consequence of Theorem \ref{th:J_G} suggests that we can expect that the marked scheme $\Mf(\mathcal P(\mathfrak j_G),3)$ contains smoothable Gorenstein points. 
\begin{proposition}
$\Mf_{\mathcal P(\mathfrak j_G),3}\cap \mathrm{Gor}{(1,5,5,1)}$ is a non-empty open subset.
\end{proposition}

\begin{proof}
From the proof of Theorem \ref{th:J_G} we deduce that the ideal $\mathrm{Ann}(F)$ belongs to the family of ideals having a $[\mathcal P(\mathfrak j_G),3]$-marked basis, hence to the family of ideals defining the $[\mathcal P(\mathfrak j_G),3]$-scheme $\Mf_{\mathcal P(\mathfrak j_G),3}$.
\end{proof}

By the techniques described in \citep{BCR}, we obtain $\Mf_{\mathcal P(\mathfrak j_G),3}$ as the affine scheme defined by an ideal $\id U_2$ generated by $576$ polynomials in the polynomial ring $K[C]$ in $204=12\cdot 17$ variables. By a suitable choice of values for the parameters $C$, we find a family $\{\id G_T\}_T$ of $[\mathcal P(\mathfrak j_G),3]$-marked bases consisting of the following polynomials whose coefficients depend on the parameter $T$: 

\noindent $F_1:=\id f_1$, 
$F_2:=\id f_2$, 
$F_3:=\id f_3-Tx_4+x_2T$,
$F_4:=\id f_4$,
$F_5:=\id f_5$, 
$F_6:=\id f_6$,
$F_7:=\id f_7$,
$F_8:=\id f_8$,\\
$F_9:=\id f_9$,
$F_{10}:=\id f_{10}$,
$F_{11}:=\id f_{11}$,
$F_{12}:=\id f_{12}-x_2 x_3 T-x_3 x_1 T+T x_2^2+x_2 x_1 T$,
$F_{13}:=\id f_{13}$,\\
$F_{14}:=\id f_{14}$,
$F_{15}:=\id f_{15}$,
$F_{16}:=\id f_{16}$,
$F_{17}:=\id f_{17}.$

\noindent From now, for every $T\in \mathbb A^1_K$ we denote by $\id i_{T}\subset K[x_1,\dots,x_{7}]$ the ideal generated by the $[\mathcal P(\mathfrak j_G),3]$-marked basis $\id G_{T}$. For $T=0$ we obtain the ideal $\id i_0=\mathrm{Ann}(F)$ that we considered in the proof of Theorem \ref{th:J_G} and which defines a scheme with support in a single point. For every $T\in \mathbb A^1_K\setminus\{0\}$ we have a different situation because every ideal $\id i_{T}$ defines a scheme whose support contains at least the following $3$ distinct affine points: 
$$(0,0,0,T,0),\qquad (0,0,0,0,0),\qquad(0,-T,0,0,0).$$

\begin{lemma}\label{lemma:piattezza5}
Over a field $K$ of characteristic $0$, there is a flat family of ideals which is contained in $\mathrm{Mf}_{\mathcal P({\mathfrak j_G}),3}\cap \mathcal R_{12}^5$ such that the special fiber corresponds to a Gorenstein point defined by a graded $K$-algebra with Hilbert function $(1,5,5,1)$.
\end{lemma}

\begin{proof} We prove that the family of ideals $\{\id i_{\tau}\}_{\tau}$, which are generated by the $[\mathcal P({\mathfrak j_G}),3]$-marked bases $\id G_{\tau}$, is contained in the smoothable component $\mathcal R_{12}^5$. All the computations we refer to can be repeated using the polynomials and values stored in the ancillary files available  in the folder ``Lemma \ref{lemma:piattezza5}''.

By construction, the ideals $\id j_{T}$ belong to the marked scheme $\Mf_{\mathcal P(\id j_G),3}$ which embeds as an open subset in the punctual Hilbert scheme $\hilb_{12}^5$. Hence, these ideals define a flat family over $\mathbb A^1$. As we have already recalled in Section \ref{sec:gore}, the locus  of points in a Hilbert scheme representing all the Gorenstein schemes is an open subset. Hence, the intersection of this locus with the flat family $\{\id i_{T}\}_{T}$, which is non-empty because $\id j_0$ represents a Gorenstein point, is an open subset of the family. So, we find at least a value $\overline T\neq 0$ such that $\id i_{\overline T}$ represents a Gorenstein point.

By computational tools, we find that for every $T\neq 0$ the ideal $\id i_{T}$ defines a scheme whose supports contain at least $3$ distinct points and hence components of multiplicity at most $10$. So, the ideal $\id i_{\overline T}$ is smoothable by the fact that for $d\leq 10$ the locus of Gorenstein points is a Hilbert scheme of $d$ points is irreducible (see \citep{CN10} together with \citep[Proposition 1.2]{JJB2017}). Thus, every ideal $\id i_T$ belongs to the smoothable component $\mathcal R_{12}^5$ because the family is irreducible. In particular, the limit of this family, that is the ideal $\id j_0$, belongs to $\mathcal R_{12}^5$ too.
\end{proof}

\begin{remark} 
We now highlight the following fact, which is analogous to that described in Remark \ref{rem:fuori schema}. The ideal $\id i_0$ is the ideal $\mathrm{Ann}(F)$ of the proof of Theorem \ref{th:J_G}, hence it defines a Gorenstein point in $\mathrm{Gor}(1,5,5,1)$. 
Nevertheless, for every $T\neq 0$, the ideal $\id i_{T}$ defines a Gorenstein point which does not belong to $\mathrm{Gor}(1,5,5,1)$ because it is supported on more than one point. Moreover, in the polynomial $\id f_{14}$ the term $x_3x_2$ has a non-null coefficient and is higher than $x_2^3$ with respect to lex term order. This fact implies that the initial ideal of $\id i_{T}$ with respect to lex order is not $\mathfrak j_G$.  
\end{remark}

\begin{lemma}\label{lemma:liscezza5}
Over a field $K$ of characteristic $0$, there is a smooth  Gorenstein point defined by a graded $K$-algebra with Hilbert function $(1,5,5,1)$ belonging to the smoothable component $\mathcal R_{12}^5$. 
\end{lemma}

\begin{proof}
We prove that the ideal $\id i_0$ defines a smooth Gorenstein point in the smoothable component $\mathcal R_{12}^5$. All the computations we refer to can be repeated using the polynomials and values stored in the ancillary files available  in the folder ``Lemma \ref{lemma:liscezza5}''.

 We already know that  $\id i_0$ defines a Gorenstein point. Moreover, by Lemma \ref{lemma:piattezza5} and by construction, the ideal $\id i_0$ represents a point that belongs to $\mathcal R_{12}^5\cap \Mf_{\mathcal P({\mathfrak j_G}),3}$. Due to Corollary \ref{cor:tangente},  we can compute the tangent space from the polynomials of its $[\mathcal P(\mathfrak j_G),3]$-marked basis obtaining that the dimension of the Zariski tangent space to $\hilb^5_{12}$ at the point $\id j_0$ is $60$, i.e.~the dimension of the smoothable component.
\end{proof}

\begin{theorem} \label{Gor5}
Over a field $K$ of characteristic $0$, every Gorenstein point defined by a local $K$-algebra with Hilbert function $(1,5,5,1)$ is smoothable.
\end{theorem}

\begin{proof}
By Lemmas \ref{lemma:piattezza5} and \ref{lemma:liscezza5}, we have a Gorenstein point of $\mathrm{Gor}(1,5,5,1)$ that belongs to the smoothable component $\mathcal R_{12}^5$ and that is smooth in the Hilbert scheme. These facts imply that also all the other points of $\mathrm{Gor}(1,5,5,1)$ belong to $\mathcal R_{12}^5$, i.e. are smoothable, because the locus $\mathrm{Gor}{(1,n,n,1)}$ is irreducible, as we have already recalled. Finally, we conclude due to \citep[Proposition 1.1]{JJ2017}.
\end{proof}

\section*{Acknowledgments}

We thank Joachim Jelisiejew for kindly providing useful knowledge about the equivalence between graded and local Gorenstein Artin $K$-algebras with Hilbert function of type $(1,n,n,1)$ and about the smoothability of a local Artin $K$-algebra over any field. We also thank the anonymous referee for very useful comments and suggestions, in particular for highlighting the existence of the two irreducible components of $\hilb_{16}^7$ that can be constructed over the elementary component of $\hilb_8^7$. We are grateful to Steven Kleiman for several useful discussions on this paper.

The authors are members of GNSAGA (INdAM). The second author is partially supported by Prin 2015 (2015EYPTSB\_011 - {\em Geometry of Algebraic varieties}, Unit\`{a} locale Tor Vergata, CUP  E82F16003030006).
Part of the research was conducted during a visit of the first and third author at the Dipartimento di Matematica e Applicazioni of Universit\`a  di Napoli Federico II, which also financially supported that visit.

\end{document}